\documentclass{article}
\usepackage{amsmath}
\usepackage{amsthm}
\usepackage{amssymb}
\usepackage{epstopdf}
\usepackage{graphicx}
\newtheorem{theorem}{Theorem}
\newtheorem{corollary}{Corollary}

\begin{document}
\title
{Sum Rules for Functions of the Riemann Zeta Type}
\author{R.C. McPhedran,\\
School of Physics,\\
University of Sydney}
\maketitle
\begin{abstract}
We consider analytic functions of the Riemann zeta type, for which, if $s$ is a zero, so is $1-s$. We use infinite product representations
of these functions, assuming their zeros to be of first order. We use exponential factors to accelerate convergence, and by comparing the exponential factors with the Taylor series coefficients about $s=0$ of the original function, connect these coefficients with sums of powers of reciprocals of the zeros, in the form of sum rules. Such sum rules have been previously considered by Lehmer and Keiper, but the approach and applications taken here are more general. In related work, a new sufficient condition is found for the Riemann hypothesis, and the basis for this condition is discussed.
\end{abstract}

\section{Introduction}
This paper consists of two almost independent parts. Sections 2-4 study sum rules for inverse powers of analytic functions related to the Riemann zeta function. This work builds on and generalises that in papers by Lehmer and Keiper, and was motivated by the hope that something more could be said about the Riemann hypothesis based on those related functions for which the Riemann hypothesis has been proved to hold, together with their connection to the zeta function. Technical difficulties as yet unresolved have left this aim unachieved in the work reported here. Sections 5 and 6 again  take up the connection between the Riemann hypothesis for the related functions and for the zeta function. The result is a new sufficient condition for the validity of the Riemann hypothesis, which it is hoped will be of interest to other workers interested in the question. An argument is put forward for the validity of the new condition.

\section{The Product Representation}
We use the following result from Whittaker and Watson \cite{wandw}:
\begin{theorem}
Let $f(z)$ be a function analytic for all values of $z$, with simple zeros at the points $a_0$, $a_1$, $a_3$, $\ldots$ which tend to infinity as $n\rightarrow \infty$, with $|a_n|\neq 0$ for all $n$.  Suppose we can find a set of circles $C_m$ such that $f'(z)/f(z)$ is bounded on them as $m\rightarrow\infty$. Then $f(z)$ has the following infinite product representation:
\begin{equation}
f(z)=f(0) e^{f'(0) z/f(0)}\prod_{n=1}^\infty \left\{\left(1-\frac{z}{a_n}\right)e^\frac{z}{a_n}\right\}.
\label{wwprod}
\end{equation}
\end{theorem}

The exponential factor in (\ref{wwprod}) will be recognised as necessary to ensure convergence. In our case of interest, we may divide the zeros up into a group denoted $a_n$, and a second group denoted $a_{-n}=1-a_n$. We then have
\begin{equation}
(1-\frac{z}{a_n}) (1-\frac{z}{1-a_n})=1-\frac{z(1-z)}{a_n (1-a_n)}.
\label{prod2}
\end{equation}
The infinite product will thus converge absolutely provided $a_n (1-a_n)$ scales faster than $n$ as $n$ goes to infinity. We assume this to be the case: it holds for functions whose zeros behave similarly to those of the Riemann zeta function. The proposition was proved for the function $\xi(s)$ to be considered in the next section by Jacques Hadamard \cite{had}.

Continuing then to study the form of the expansion (\ref{wwprod}) for zeta-like functions, we generalise the convergence factor by using the exponential series to give convergence of the product at a rate faster than  $1/a_n^2$ :
\begin{equation}
\left\{\left(1-\frac{z}{a_n}\right)e^\frac{z}{a_n}\right\}
\rightarrow \left\{\left(1-\frac{z}{a_n}\right)e^{[\sum_{m=1}^M (s^m/m)\sum_n (1/a_n^m)]} \right\} .
\label{prod3}
\end{equation}
The exponential factor will be recognised as having M terms from the series for $-\log (1-z/a_n)$.
We then construct two expansions for zeta-like functions: one is based on the Taylor series for $\log (f(z)/f(0))$:
\begin{equation}
f(z)=f(0) \exp{\left[\sum_{m=1}^M\left.\frac{d^m}{d s^m} \log f(z)\right|_{s=0} \frac{s^m}{m!}\right]}
\exp{\left[\sum_{m=M+1}^\infty \left.\frac{d^m}{d s^m} \log f(z) \right|_{s=0} \frac{s^m}{m!}\right]}
\label{prod4}
\end{equation}
The second is based on the product representation:
\begin{equation}
f(z)=f(0) \exp \left[-\sum_{m=1}^{M} \sum_n \frac{s^m}{m a_n^m} \right]\prod_{n=1}^\infty  \left\{\left(1-\frac{z}{a_n}\right)\exp {\left[\sum_{m=1}^M \frac{s^m}{m}\sum_n \frac{1}{a_n^m}\right]} \right\}  .
\label{prod5}
\end{equation}
Comparing (\ref{prod4}) and (\ref{prod5}), we arrive at an infinite set of sum rules for the zeros of $f(z)$:
\begin{equation}
\left. \frac{1}{m!}\frac{d^m}{d z^m} \log f(z)\right|_{z=0} =-\frac{1}{m}\sum_n\frac{1}{a_n^m}.
\label{prod6}
\end{equation}

Another route to the same results given in equation (\ref{prod6}) is  to take the product representation without any convergence factors:
\begin{equation}
f(z)=f(0) \prod_{n=1}^\infty \left(1-\frac{z}{a_n}\right),
\label{prod7}
\end{equation}
take its logarithm, and compare the Taylor series about $s=0$ of the left-hand side with the combined  expansions of the terms
$\log (1-z/a_n)$ on the right-hand side. The method of the preceding paragraph seems likely to be valid in a wider variety of cases than this alternative.

For an extensive discussion of sum rules and their applications in other contexts, see King \cite{king}.
\section{Examples of Sum Rules}

{\bf Remark:} It should be noted that the two sides of the identity (\ref{prod6}) are complementary in computations. The left-hand side is most easily evaluated when $m$ is small, while the right-hand side will then need the largest set of zeros for accuracy. Conversely, when $m$ increases, the accurate calculation of multiple-order derivatives becomes harder, while the necessary set of zeros becomes smaller.

\subsection{$\xi (s)$}
The function $\xi(s)$ is even under $s\rightarrow 1-s$ and is defined as
\begin{equation}
\xi(s)=\frac{1}{2}  s(s-1) \frac{\Gamma (s/2)\zeta (s)}{\pi^{s/2}}.
\label{xi1}
\end{equation}
That all its zeros lie on the critical line is the Riemann hypothesis, while extensive numerical investigations have not encountered any zeros of multiple order. The data set of zeros we rely on here has 10,000 elements ranging up to $t=9877.78$, and was generated in Mathematica. 
Using this data set, a few low-order examples of values of the left-hand side of (\ref{prod6}) provided by Mathematica and the right-hand side provided by direct summation are given in Table \ref{tab1}. The results  for $m=1,2$ show that the set of 10,000 zeros is insufficient to provide accurate results, but the accuracy improves rapidly from $m=3$ on. The evaluation of higher-order derivatives 
in Mathematica tends to become slower as $m$ increases. With use of only 100 zeros from the data set, a good cross-over choice for
$m$ if accuracy of 12 decimals is sought is 8: below that, values should be provided by direct differentiation in Mathematica, and above that from the set of 100 zeros. The difference in results for $m=8$ is $9.27 \times 10^{-16}$.

\begin{table}
\label{tab1}
\begin{tabular}{|c|c|c|}\hline
$m$ &LHS (\ref{prod6}) & RHS (\ref{prod6}) \\ \hline
1 & -0.0230957 & -0.022961  \\
2& 0.0230772 & 0.0229425 \\
3& 0.0000370527 & 0.0000370527\\
4& -0.0000184068 & -0.0000184068\\
5 & $-1.43019 \times 10^{-7}$ & $-1.43019 \times 10^{-7}$ \\
6& $ 4.69061\times 10^{-8}$ & $ 4.69061\times 10^{-8}$\\
\hline
\end{tabular}
\caption{Numerical examples of the sum rule (\ref{prod6}) for the function $\xi (s)$.}
\end{table}

We now quote some results given by Keiper \cite{keiper} for various expansions linked with the function $\xi (s)$, which we will generalise subsequently to three other functions.  Keiper makes the definition for the expansion of $\xi (s)$ about $s=1$:
\begin{equation}
\frac{\xi'(s)}{\xi(s)}=\sum_{k=0}^\infty \sigma^K_{k+1} (1-s)^k,
\label{consist1}
\end{equation}
where we have added a superscript $K$ to Keiper's $\sigma_k$. Given the symmetry of $\xi(s)$ under $s\rightarrow 1-s$, 
the expansion of $\xi(s)$ about $s=0$ is, from  (\ref{consist1}):
\begin{equation}
-\frac{\xi'(s)}{\xi(s)}=\sum_{k=0}^\infty  \sigma^K_{k+1} s^k
\label{consist2}
\end{equation}
We compare this with the result (\ref{prod6}), which for this case is
\begin{equation}
\log\left( \frac{\xi(s)}{\xi(0)}\right)=- \sum_{k=1}^\infty \left(\frac{1}{k}\right) \sigma_k^M s^k,
\label{consist3}
\end{equation}
denoting the sum over zeros in this equation by $\sigma_k^M $. Taking the derivative of (\ref{consist3}), we arrive at the consistency
of the two approaches:
\begin{equation}
\sigma_k^M =\sigma_k^K ~{\rm for all}~ k.
\label{consist4}
\end{equation}
Hence, the superscripts can be dropped.


Keiper denotes the zeros of $\xi (s)$ by $\rho$, so that (\ref{consist1}) gives
\begin{equation}
\sigma_k=\sum_{\rho} \frac{1}{\rho^k}.
\label{ke2}
\end{equation}
Keiper further derives two sum rules from the functional equation for $\xi(s)$:
\begin{equation}
\sum_{k=1}^\infty  \frac{1}{k}\sigma_k=0 ,
\label{ke3}
\end{equation}
and
\begin{equation}
\sigma_1=-\sum_{k=1}^\infty  \sigma_k .
\label{ke4}
\end{equation}
He also gives a recurrence relation:
\begin{equation}
\sigma_{j+1}=(-1)^{j+1}\sum_{k=1}^\infty \left(\begin{tabular}{c}
$k-1$\\
$j$
\end{tabular}\right)  \sigma_k
\label{ke5}
\end{equation}

Keiper also consider the relationship between the coefficients $\sigma_k$ and those occurring in two further expansions:
\begin{equation}
\frac{\xi'(1/s)}{\xi(1/s)}=\sum_{k=0}^\infty \tau_{k} (1-s)^k ,
\label{ke6}
\end{equation}
and
\begin{equation}
\log (2\xi(1/s)) =\sum_{k=0}^\infty \lambda _{k} (1-s)^k.
\label{ke7}
\end{equation}
He shows that:
\begin{equation}
\tau_0=\sigma_1,
\label{ke8}
\end{equation}
and
\begin{equation}
\tau_k=\sum_{j=1}^k \left(\begin{tabular}{c}
$k-1$\\
$j-1$
\end{tabular}\right) (-1)^j \sigma_{j+1} ~{\rm for} ~ k\ge 1.
\label{ke9}
\end{equation}
Also,
\begin{equation}
\lambda_0= 0,
\label{ke10}
\end{equation}
and
\begin{equation}
\lambda_k=\sum_{j=1}^k \frac{(-1)^{j-1}}{j} \left(\begin{tabular}{c}
$k-1$\\
$j-1$
\end{tabular}\right) (-1)^j \sigma_{j} ~{\rm for} ~ k\ge 1.
\label{ke11}
\end{equation}
Note that equation (\ref{ke10}) is correct, even though the multiple precision tables of these coefficients in the Appendix to Keiper's paper ascribes a  non-zero value to $\lambda_0$. (Keiper also derives the connection between the $\sigma_k$ and the
Stieltjes constants $\gamma_k$, but we will not pursue this topic here.)

The coefficients $\tau_k$ hold a particular interest in relation to the Riemann hypothesis. Indeed, as Keiper shows,
\begin{equation}
\tau_{m-1}=-\sum_\rho  \left(\frac{\rho}{\rho-1}\right)^m \rho^{-2}.
\label{ke12}
\end{equation}
From (\ref{ke12}), if the Riemann hypothesis holds, the $|\tau_k|$ must be bounded by 
\begin{equation}
\sum_\rho |\rho|^{-2}=0.046191479322\ldots .
\label{ke13}
\end{equation}
On the other hand, if the $|\tau_k|$ are bounded, then for no $\rho$, $|\rho|>|1-\rho|$, so the Riemann hypothesis holds.

In terms of the behaviour of the $\lambda_k$, Li's criterion \cite{li, bomblag} states that the Riemann hypothesis is equivalent to 
$\lambda_k\ge 0$ for every positive integer $k$. These two criteria are illustrated  in Fig. \ref{fig-xi1}. The quantities $\tau_k$ decrease as $k$ increases in the range shown, moving further below the limit $0.046191479322$. The quantities $\lambda_k$ increase roughly linearly with $k$ (the slope being about $0.023$), again moving away from the limit of zero. Keiper comments on the difficulty of finding numerically exceptions to the Riemann hypothesis using this sort of behaviour.
 Indeed, if we consider the equation
 \begin{equation}
 \lambda_m=\frac{1}{m}\sum_\rho \left[ 1- \left(\frac{\rho}{\rho-1}\right)^m\right],
 \label{ke14}
 \end{equation}
 then for the quantity in square brackets to become negative for an exception to the Riemann hypothesis with $t>T$,
 we require (roughly) $m>2 T^2$. Currently, $T=O(10^9)$, so $m>O(10^{18})$.
\begin{figure}[tbh]
\includegraphics[width=6 cm]{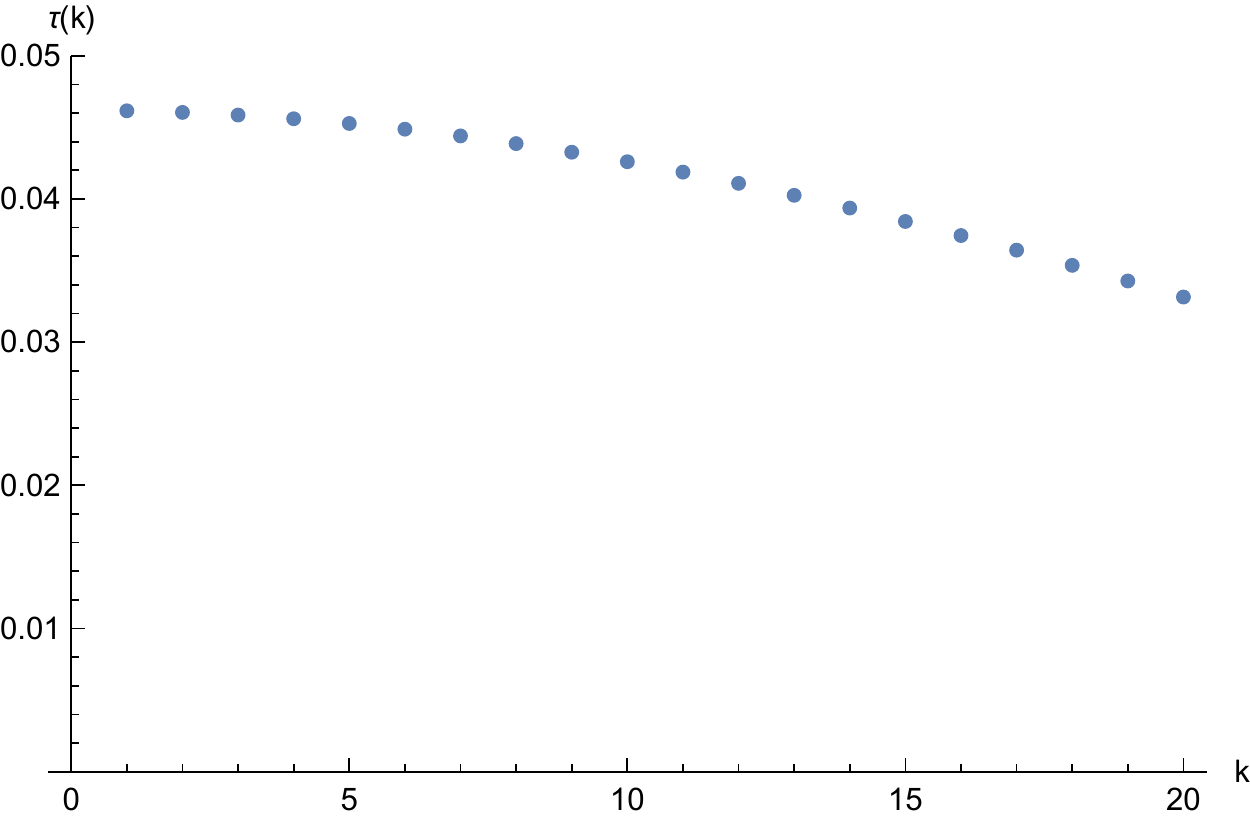}~~\includegraphics[width=6 cm]{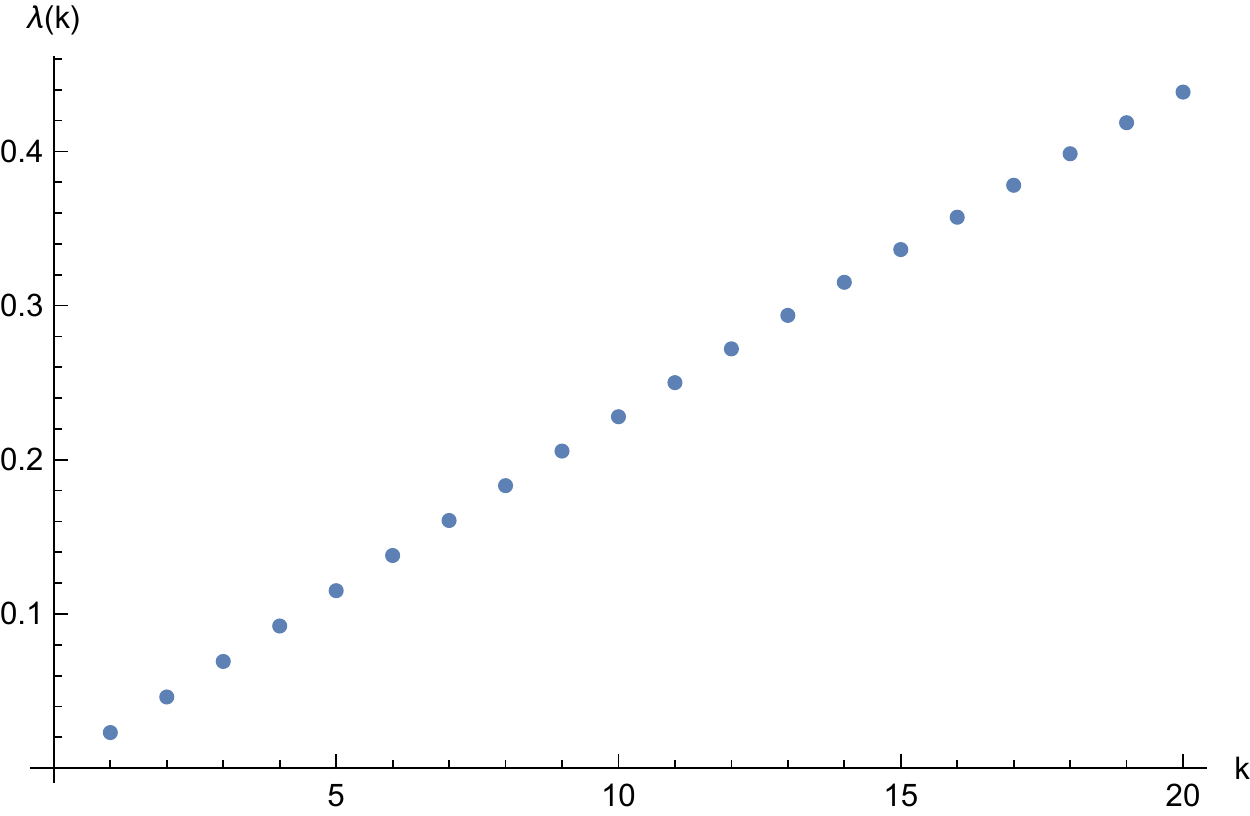}
\caption{(Left)} The coefficients $\tau (k)$ of $\xi (s)$ as a function of their order $k$. (Right) The coefficients $\lambda (k)$ as a function of $k$.
\label{fig-xi1}
\end{figure}

\subsection{${\cal T}_+(s)$}
We continue with two examples related to functions for which it is known that all non-trivial zeros are first-order and located
on the critical line \cite{ki,lagandsuz,mcp13}.

The first  function is defined as:
\begin{equation}
{\cal T}_+(s)=\frac{1}{4} [\xi_1(2 s)+\xi_1(2 s-1)].
\label{T+0}
\end{equation}
It has a pole of order unity at $s=0$:
\begin{equation}
{\cal T}_+(s)\sim-\frac{1}{8 s}+\frac{1}{24}(3\gamma+\pi-3\log(4\pi))+O(s).
\label{T+1}
\end{equation}
It tends to a constant at $s=1/2$:
\begin{equation}
{\cal T}_+(s)\sim \frac{1}{4} (\gamma-\log (4\pi))+O((s-1/2)^2).
\label{T+2}
\end{equation}
It has a pole of order unity at $s=1$:
\begin{equation}
{\cal T}_+(s)\sim \frac{1}{8 (s-1)}+\frac{1}{24}(3\gamma+\pi-3\log(4\pi))+O(s).
\label{T+3}
\end{equation}
It is even under $s\rightarrow 1-s$.

The function ${\cal T}_+(s)$ takes the following form on the critical line:
\begin{equation}
{\cal T}_+(1/2+i t)=2 |\xi_1(1+2 i t)| \cos[\arg(\xi_1(1+2 i t)|],
\label{T+CL}
\end{equation}
and thus its zeros correspond to $\arg(\xi_1(1+2 i t)|=(n+1/2)\pi$ for  any integer $n$.

The author has compiled a list of the first 1517 zeros of ${\cal T}_+(s)$, the last of which is at $t=\Re(s)\simeq 999.912$. (A copy of this may be obtained from the author.) This was used in the following numerical study.

In order to apply the results of the previous section, we consider the function with its poles eliminated:
\begin{equation}
{\tilde{\cal T}}_+(s)=s(1-s) {\cal T}_+(s).
\label{T+4}
\end{equation}
Truncated expansions  for ${\tilde{\cal T}}_+(s)$ about $s=0$ are available from Mathematica, either in symbolic or numeric form
(with the former rapidly increasing in complexity with increasing order). Numerical results obtained with the set of zeros mentioned are given in Table \ref{tab2}.
Note that in cases where all zeros lie on the critical line, we have on the right-hand side of (\ref{prod6}) a sum of the form
\begin{equation}
\sum_n\frac{1}{a_n^m}=\sum_{n>0} \left(\frac{1}{a_n^m}+\frac{1}{\overline{a_n^m}} \right),
\label{T+5}
\end{equation}
which is real.

\begin{table}
\label{tab2}
\begin{tabular}{|c|c|c|}\hline
$m$ &LHS (\ref{prod6}) & RHS (\ref{prod6}) \\ \hline
1 &-0.093389 & -0.091219\\
2&0.0930802 & 0.09091 \\
3& 0.000614337 & 0.000614336\\
4&-0.000299036 & -0.000299035\\
5 &$ -9.63049\times 10^{-6}$ & $-9.63049\times 10^{-6} $ \\
6& $2.99816\times 10^{-6}$ & $2.99816\times 10^{-6}$ \\
\hline
\end{tabular}
\caption{Numerical examples of the sum rule (\ref{prod6}) for the function ${\tilde{\cal T}}_+(s)$.}
\end{table}

\begin{figure}[tbh]
\includegraphics[width=12cm]{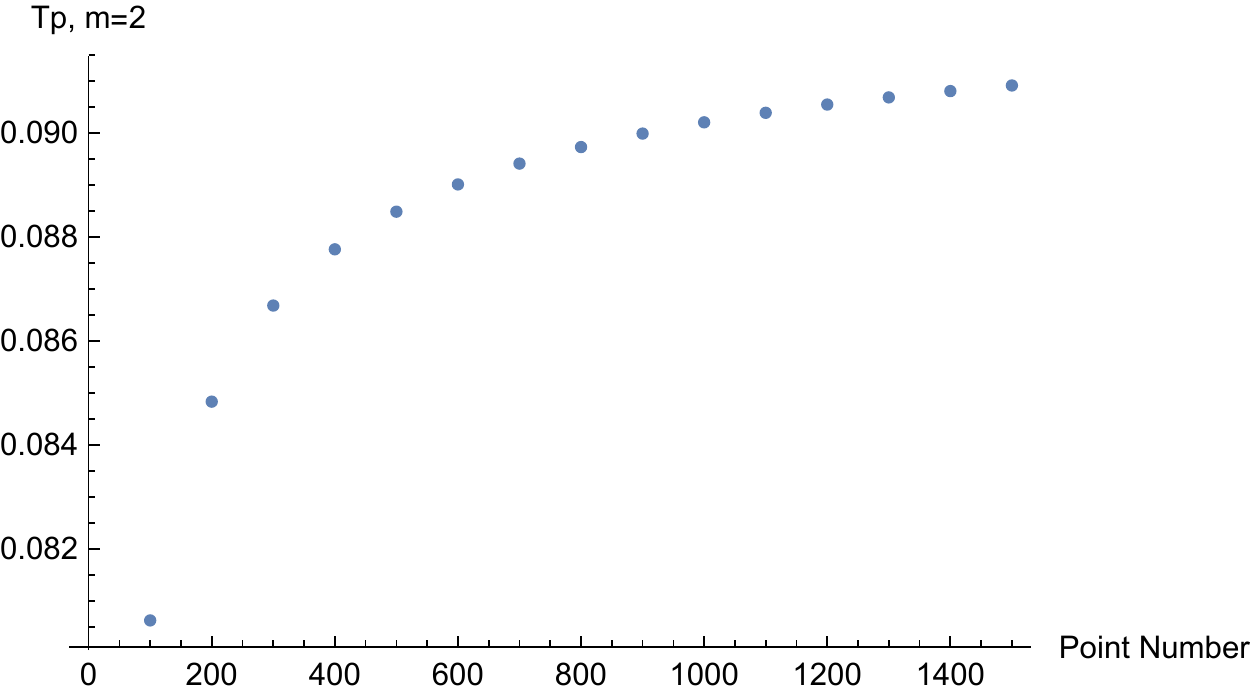}\caption{Example. of the sum rule for ${\cal T}_+(s)$ with $m=2$.}
\label{fig-Tp1}
\end{figure}

\begin{figure}[tbh]
\includegraphics[width=12cm]{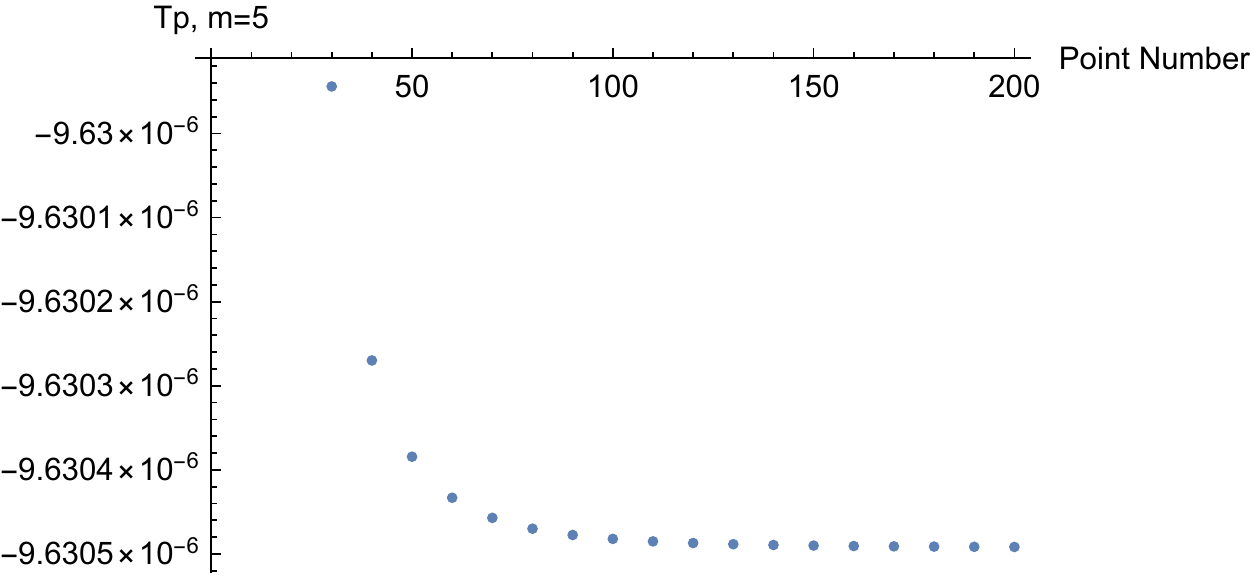}\caption{Example of the sum rule for ${\cal T}_+(s)$ with $m=5$.}
\label{fig-Tp2}
\end{figure}

Graphs illustrating the convergence are given in Figs. \ref{fig-Tp1} and \ref{fig-Tp2}. For the case $m=2$, the 1500 points in the data set is insufficient to yield high  accuracy, but for $m=5$ many fewer points can give good accuracy.

In order to implement the calculation of the elements in Tables \ref{tab1} and \ref{tab2}, for the left-hand column the series coefficients in the expansion of ${\cal T}_+(s)$ were evaluated round $s=0$. These were used in the denominator of the quotient  ${\cal T}_+'(s)/ {\cal T}_+(s)$, with the numerator form being obtained from the denominator by multiplication with their order. The series coefficients of the quotient were then obtained from Mathematica. For the right-hand column, the zeros  denoted $\rho_+$ were obtained from the list referred to above. (An alternative way of proceeding is to evaluate the expansion of  ${\cal T}_+(s)$ around $s=1$, divide out its negative term of order zero, then take the expansion in Mathematica of the logarithm of the result.)

Note that the equivalent expansion to (\ref{consist3}) for ${\cal T}_+(s)$ is:
\begin{equation}
\log\left( \frac{{\cal T}_+(s)}{{\cal T}_+(0)}\right)=- \sum_{k=1}^\infty \left(\frac{1}{k}\right) \sigma_k^+ s^k .
\label{consist3Tp}
\end{equation}

Given the ability to evaluate the coefficients $\sigma_k^+$ in two ways, a table comparing both as a function of $k$ was formed, and the difference between the two methods was studied. As we expect the accuracy of the terms in the left-hand column to ultimately worsen for large $j$, and that of the right-hand column to continue to improve with increasing $k$, we can choose to place the cross-over from one to the other at the point of minimum absolute difference. For summation over all the zeros in the listing, this gives $k=7$ as the cross-over, with an absolute difference of $1.505 \times 10^{-13}$. Actually, at $k=7$ the difference between the values
with all elements of the list and only 100 is just  $3.692 \times 10^{-15}$.

Using this method, we have verified that for the coefficients  the two sum rules hold:
\begin{equation}
\sum_{k=1}^\infty  \frac{1}{k}\sigma_k^+=0 ,~\sigma_1^+=-\sum_{k=1}^\infty  \sigma_k ^+.
\label{T+6}
\end{equation}
The recurrence relation (\ref{ke5}) is also valid for the $\sigma_k^+$.

As we have seen in the previous sub-section, an important quantity is the sum of $1/|\rho^+|^2$. We investigate this in Fig. \ref{fig-Tp3}.
Given the sum is slowly convergent, and we have Kiefer's accurate value for $4/|\rho|^2$ (taking into account the argument $2 s$ in 
${\tilde{\cal T}}_+(s)$) we compare the two in the figure, as well as showing their difference. This enables us to better estimate
the sum, taking it from a raw value of $0.182438$ to $0.186778$. The latter is probably still slightly too small. 

\begin{figure}[tbh]
\includegraphics[width=6 cm]{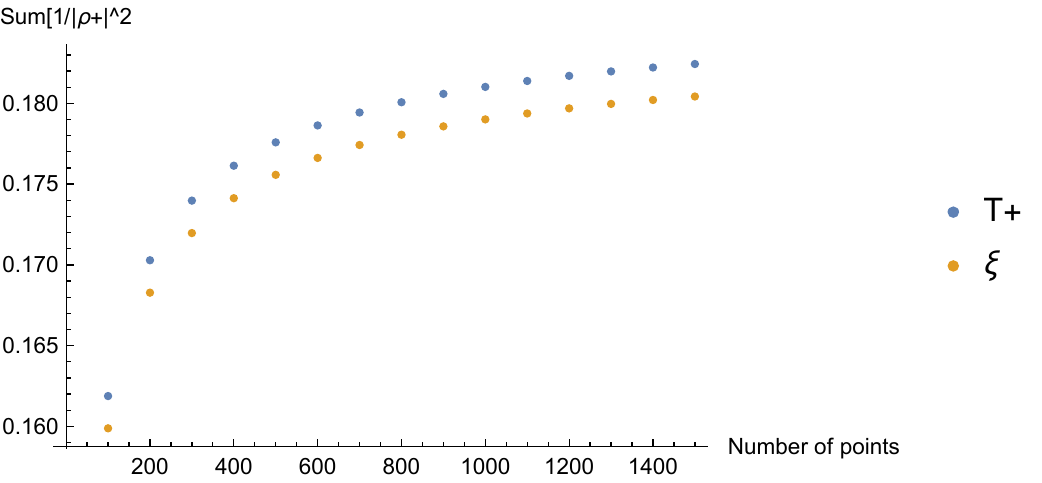}~~\includegraphics[width=6 cm]{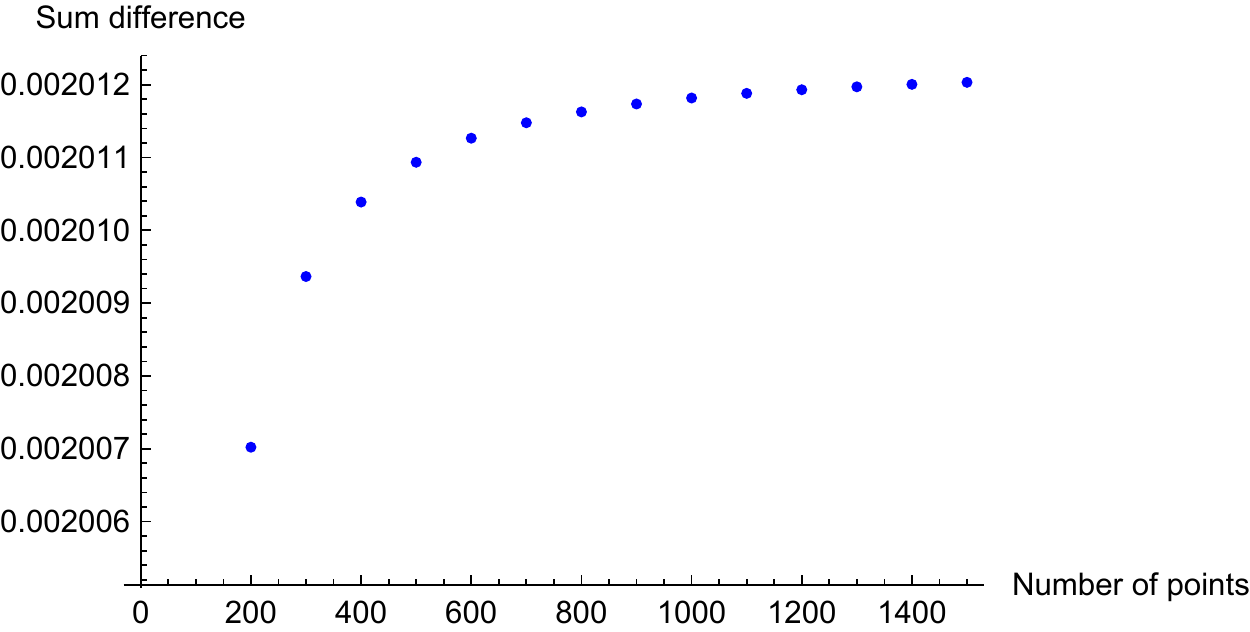}
\caption{ (Left) the sum of $1/|\rho^+|^2$ is compared with $4/|\rho|^2$; (right) the difference of these two sums.}
\label{fig-Tp3}
\end{figure}

The quantities corresponding to $\tau_k$, $\lambda_k$ for ${\cal T}_+(s)$ will be designated by a superscript "$+$". We have checked that the equations (\ref{ke7}) to (\ref{ke14}) may be used in their evaluation and the examination of their properties. The behaviour of these coefficients as a function of $k$ is shown in Fig. \ref{fig-Tp4}. As in the case of $\xi (s)$, $\tau^+_k$ decreases as $k$ increases,
while $\lambda^+_ k$ increases roughly linearly (with a slope round $0.089$, close to four times that in the previous case).

\begin{figure}[tbh]
\includegraphics[width=6 cm]{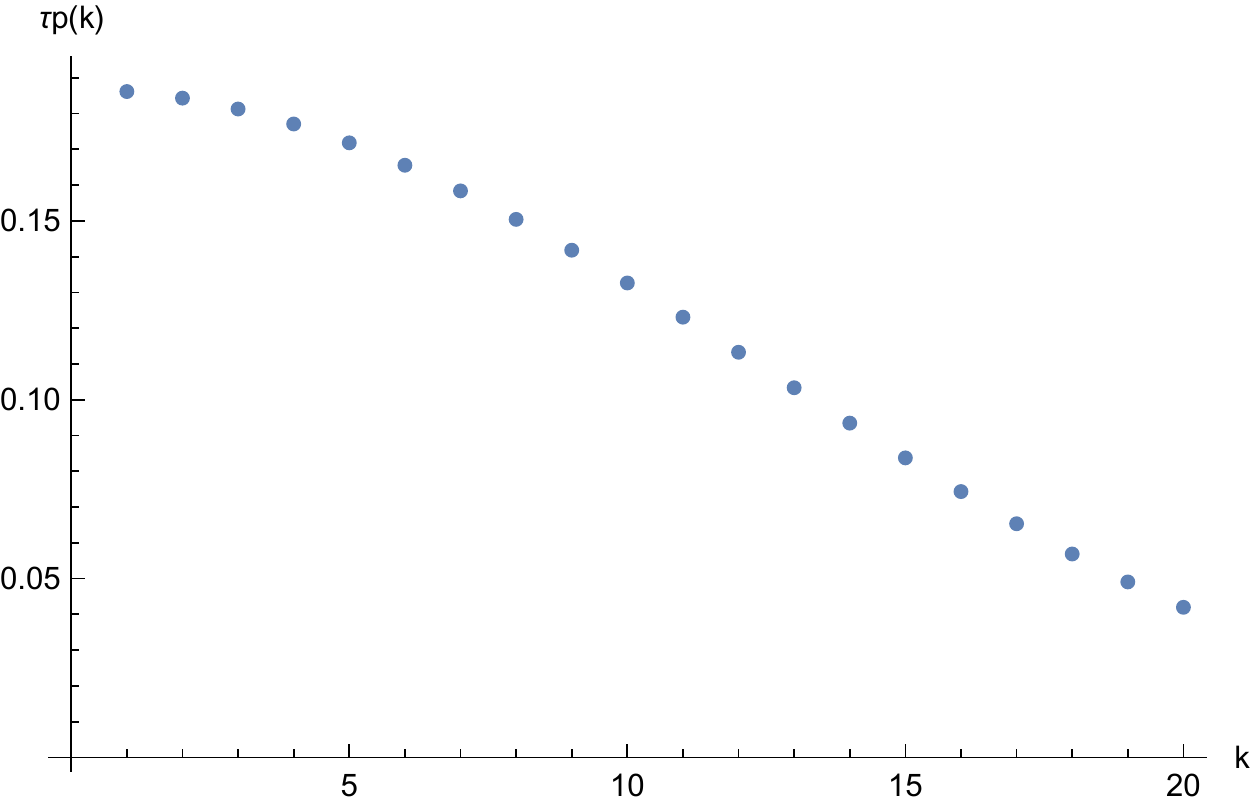}~~\includegraphics[width=6 cm]{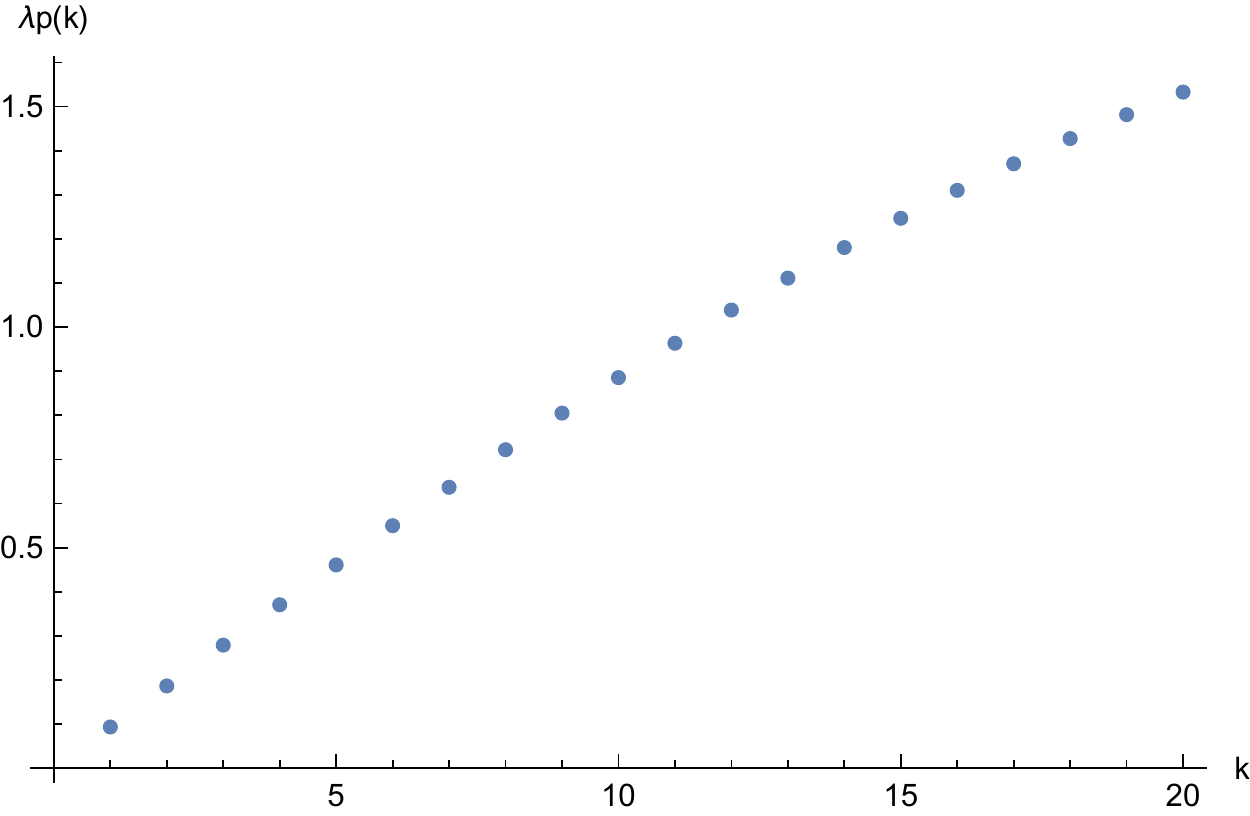}
\caption{(Left)} The coefficients $\tau^+ (k)$ of ${\tilde{\cal T}}_+(s)$ as a function of their order $k$. (Right) The coefficients $\lambda^+ (k)$ as a function of $k$.
\label{fig-Tp4}

\end{figure}
\subsection{${\cal T}_-(s)$}
This function is defined as:
\begin{equation}
{\cal T}_-(s)=\frac{1}{4} [\xi_1(2 s)-\xi_1(2 s-1)].
\label{T-0}
\end{equation}
This function is odd under $s\rightarrow 1-s$. It has poles at $s=0$, $s=1/2$ and $s=1$, and zeros on the real line at $s=3.91231$ and $s=-2.91231$. The function ${\cal T}_-(s)$ takes the following form on the critical line:
\begin{equation}
{\cal T}_-(1/2+i t)=2 i  |\xi_1(1+2 i t)| \sin[\arg(\xi_1(1+2 i t)|],
\label{T-CL}
\end{equation}
and thus its zeros correspond to $\arg(\xi_1(1+2 i t)|=n \pi$ for  any integer $n$.

We define a modified function in order to apply the sum rules (\ref{prod6}):
\begin{equation}
{\tilde{\cal T}}_-(s)=s(1-s) (s-1/2) {\cal T}_-(s).
\label{T-1}
\end{equation}
The first two terms of its series around $s=0$ are
\begin{equation}
{\tilde{\cal T}}_-(s)\sim \frac{1}{16}+\frac{1}{48} (-9 -3\gamma+\pi+\log(64 \pi^3))s+O(s^2) .
\label{T-2}
\end{equation}

The fact that all the non-trivial zeros of this function lie on the critical line was first established by P.R. Taylor, and published in a 
posthumous paper \cite{prt}. (P.R. Taylor was in fact killed on active duty with the RAF in North Africa during World War II; the paper was compiled from his notes by Mr. J.E. Rees, while the argument was revised and completed by Professor Titchmarsh.)

Table \ref{tab3} illustrates the results of numerical tests of equation (\ref{prod6}), again using a dataset of 1517 zeros of ${\cal T}_-(s)$
on the critical line running up to $t=1000$ (of course complemented by the two zeros on $t=0$ mentioned above).
It is evident that the size of the dataset is inadequate for accuracy in the case of the first two sums, but is entirely sufficient in the other four cases.

\begin{table}
\label{tab3}
\begin{tabular}{|c|c|c|}\hline
$m$ &LHS (\ref{prod6}) & RHS (\ref{prod6}) \\ \hline
1 &0.00100613 &  0.00317565\\
2& -0.00508561& -0.00725513 \\
3& 0.00838236  & 0.00838236\\
4& -0.00476457 & -0.00476457\\
5 & 0.000730707 & 0.000730707 \\
6& -0.000317834 & -0.000317834 \\
\hline
\end{tabular}
\caption{Numerical examples of the sum rule (\ref{prod6}) for the function ${\tilde{\cal T}}_-(s)$.}
\end{table}

Graphs illustrating the convergence of the sums over zeros for $m=1$ and $m=3$ are given in Figs. 3 and 4. These again shows slow convergence in the former case, and rapid convergence in the latter. Comparison of the data in Tables \ref{tab2} and \ref{tab3} also shows that the terms in the former case go more rapidly to zero than in the latter.

\begin{figure}[tbh]
\includegraphics[width=12cm]{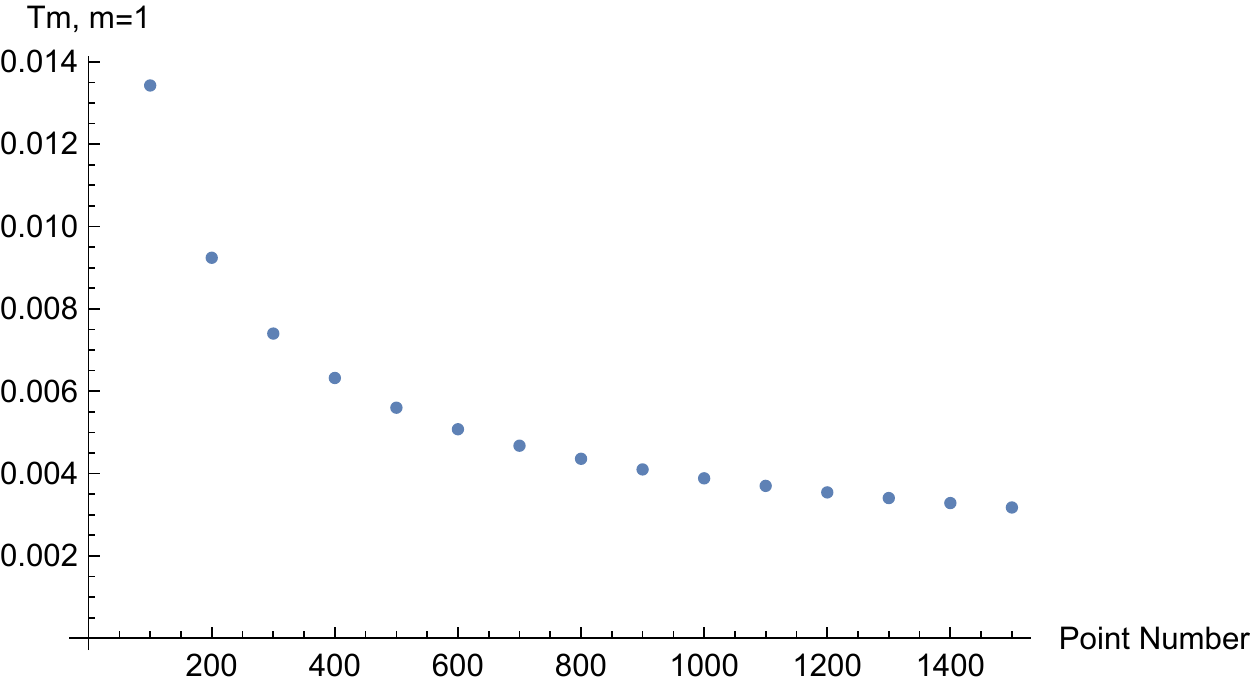}\caption{Example. of the sum rule for ${\cal T}_-(s)$ with $m=1$.}
\label{fig-Tm1}
\end{figure}

\begin{figure}[tbh]
\includegraphics[width=12cm]{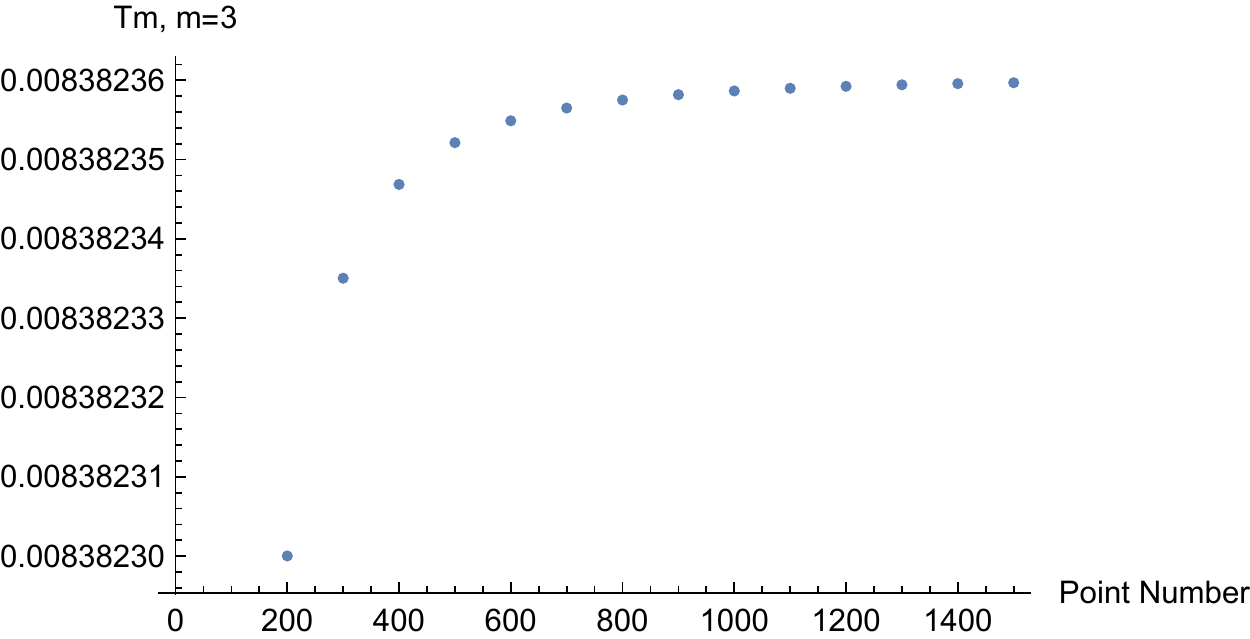}\caption{Example of the sum rule for ${\cal T}_-(s)$ with $m=3$.}
\label{fig-Tm2}
\end{figure}

The equivalents for ${\tilde{\cal T}}_-(s)$ of equations (\ref{ke3}-\ref{ke5}) have been verified. As far as the estimation of the sum over
$1/|\rho_-|^2$ is concerned, its convergence is illustrated in Fig. \ref{fig-Tm3}. The two zeros on the real axis of $s$ have not been included in the points for ${\tilde{\cal T}}_-(s)$. For comparison, the corresponding figures are shown for ${\tilde{\cal T}}_+(s)$ and $\xi(s)$, with the last of these three lying between the first two. As in the case of  ${\tilde{\cal T}}_+(s)$, we can estimate the
value of the sum for ${\tilde{\cal T}}_-(s)$ to be $0.356758$ with the  contribution of the real-axis zeros, or $0.173522$ without it.

\begin{figure}[tbh]
\includegraphics[width=12cm]{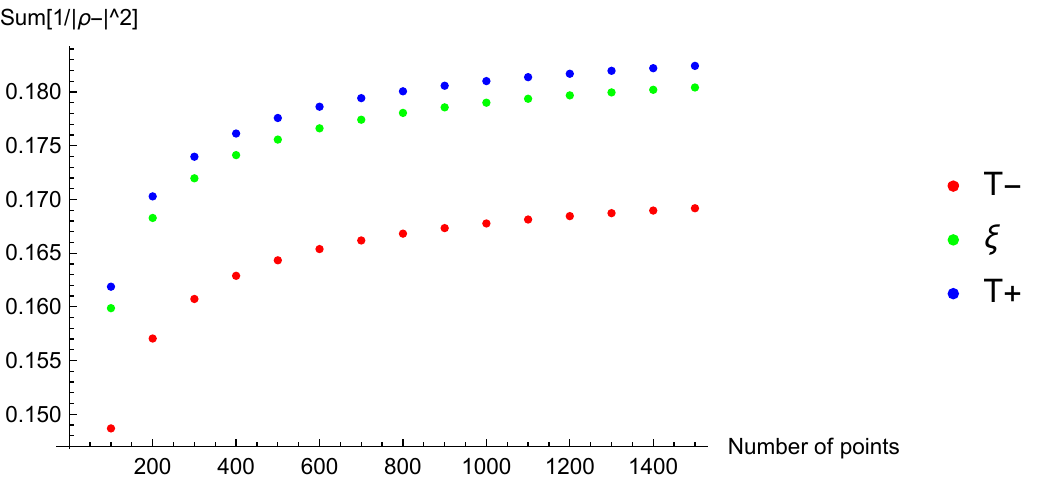}\caption{Convergence of the sum over the inverse magnitude squared of zeros for
the functions ${\tilde{\cal T}}_-(s)$,  ${\tilde{\cal T}}_+(s)$ and $\xi(s)$. Note that in the case of the first, the two off-axis zeros have not been included.}
\label{fig-Tm3}
\end{figure}

The coefficients $\tau^-_k$ and $\lambda^-_k$ for ${\cal T}_-(s)$ may be defined following the equations of  Keiper for $\xi(s)$.
However, their behaviour is quite different, and is dominated by the zeros off the critical line ($s=3.91231$ and $s=-2.91231$). This is illustrated in Fig. \ref{fig-Tm4} for $\tau^-_k$, where its behaviour as a function of $k$ is shown on the left, and is compared with the contribution from the two zeros
of the critical line (right). Both the $\tau^-_k$ and the $\lambda^-_k$ have negative real values.

\begin{figure}[tbh]
\includegraphics[width=7cm]{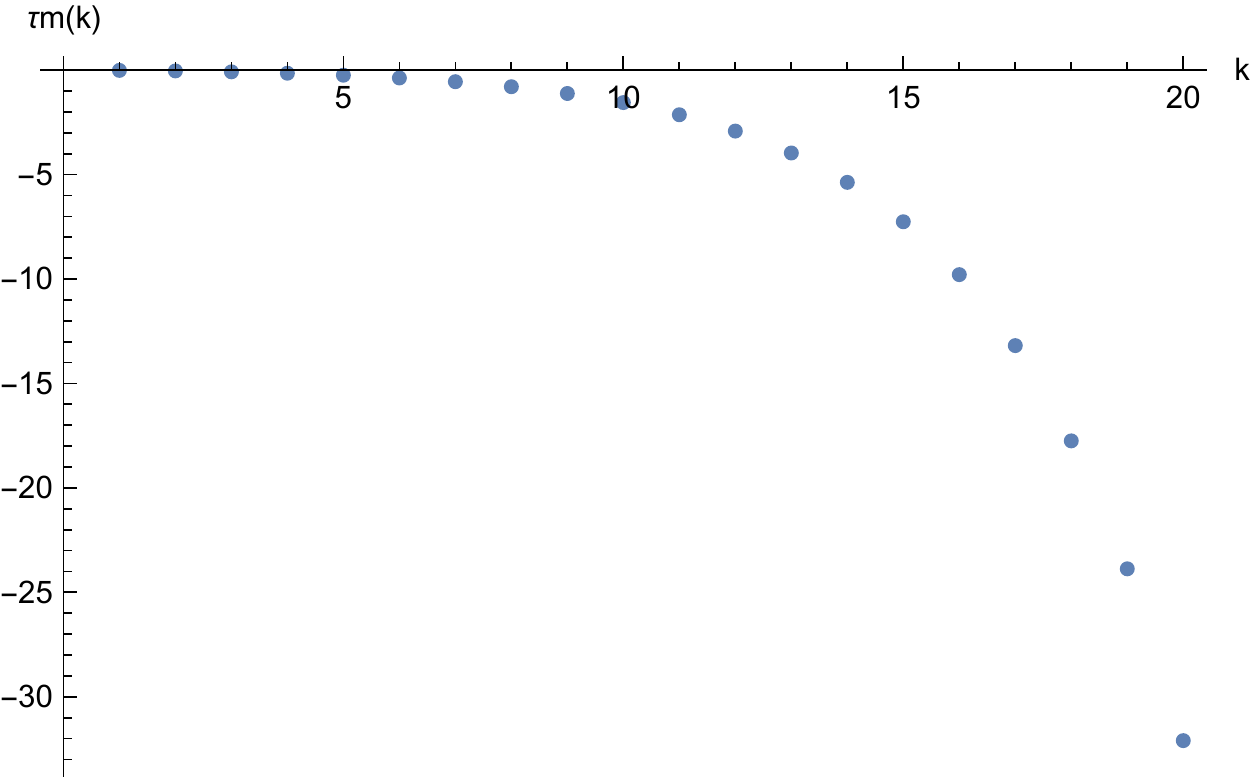}~\includegraphics[width=7cm]{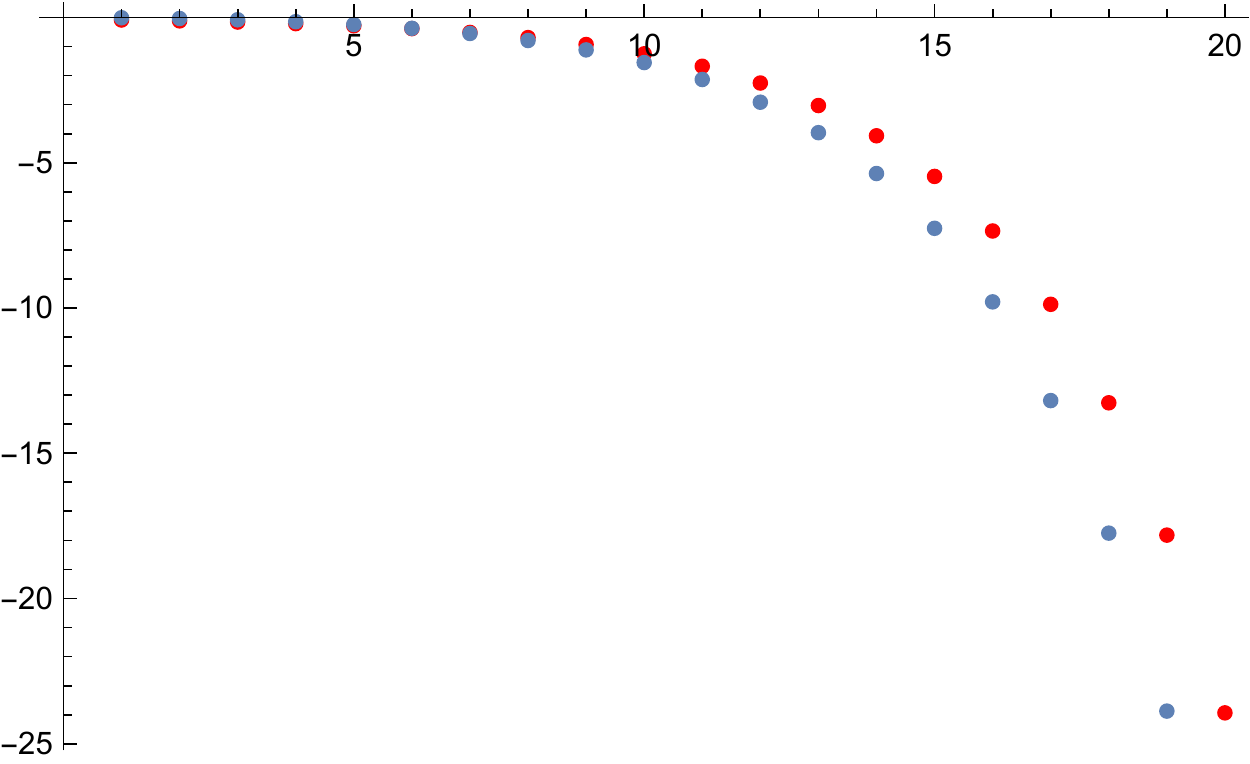}
\caption{ Behaviour of the coefficients $\tau^-(k)$ (left), and a comparison with the contribution of the two zeros of ${\cal T}_-(s)$ off the critical line (right).}
\label{fig-Tm4}
\end{figure}
\subsection{$L_{-4}(s)$}
This  Dirichlet $L$ function is  defined as a difference of two Hurwitz zeta functions:
\begin{equation}
L_{-4}(s)=\frac{1}{4^s}(\zeta(s,1/4)-\zeta (s,3/4)).
\label{lm41}
\end{equation}
It can be formed into a function even under $s\rightarrow 1-s$:
\begin{equation}
\frac{\Gamma (s) L_{-4}(s)}{\pi^{s/2}  \Gamma (s/2)}=\frac{\Gamma (1-s) L_{-4}(1-s)}{\pi^{(1-s)/2}  \Gamma ((1-s)/2)}.
\label{lm42}
\end{equation}
Its zeros all lie on the critical line $\sigma=1/2$ if the Generalised Riemann Hypothesis holds.

A set of 10,000 zeros of this function terminating at $t=1126.32039$ were used to compile the second column in Table \ref{tab4}, with the first column coming from Mathematica. Once again, the agreement with (\ref{prod6}) is excellent for $m\ge 3$.

\begin{table}
\label{tab4}
\begin{tabular}{|c|c|c|}\hline
$m$ &LHS (\ref{prod6}) & RHS (\ref{prod6}) \\ \hline
1 &-0.077784 &- 0.0776004  \\
2&0.0773251 & 0.0771415 \\
3& 0.000910626  & 0.000910626\\
4& -0.000437344 & -0.000437344\\
5 & -0.000021164 & -0.000021164 \\
6& $6.40057\times 10^{-6}$ &$6.40057\times 10^{-6}$ \\
\hline
\end{tabular}
\caption{Numerical examples of the sum rule (\ref{prod6}) for the function $L_{-4}(s)$.}
\end{table}

$L_{-4}(s)$ has no zeros off the critical line, as far as is known, and so its behaviour is similar to that of ${\tilde{\cal T}}_+(s)$ and 
$\xi(s)$.  The quantities $\sigma_k^{lm}$ have been calculated for it in the standard way, and obey the usual test relations of the kind
(\ref{ke3})-(\ref{ke7}). The convergence of the sum over the inverse squared modulus of zeros is compared in Fig. \ref{fig-Lm1}
for the three functions. This sum for $L_{-4}(s)$ is just in excess of its value for 10,000 points (0.1552)- see the graph at right in 
Fig. \ref{fig-Lm1}. Note that the density of zeros on the critical line for $L_{-4}(s)$ exceeds that for $\xi(s)$.

 \begin{figure}[tbh]
\includegraphics[width=7cm]{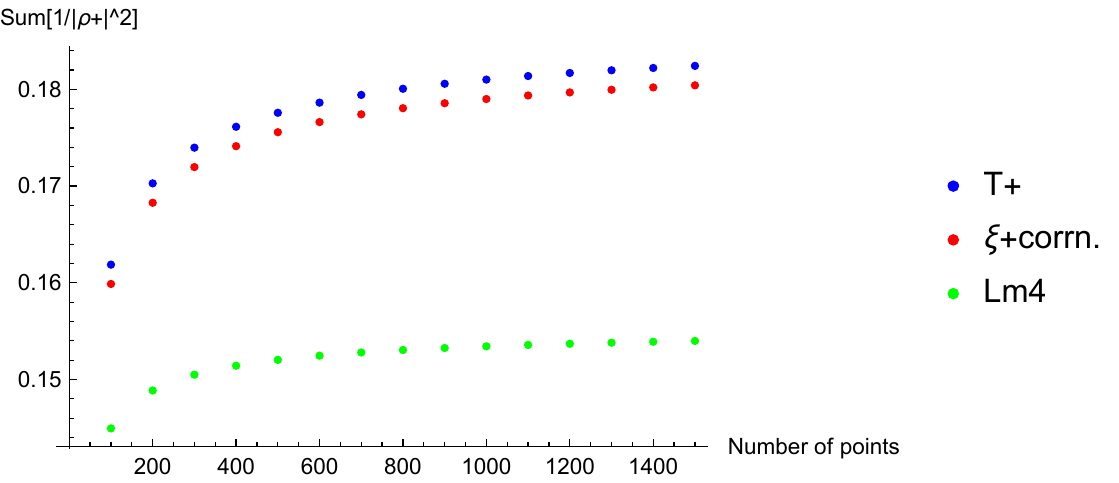}~\includegraphics[width=7cm]{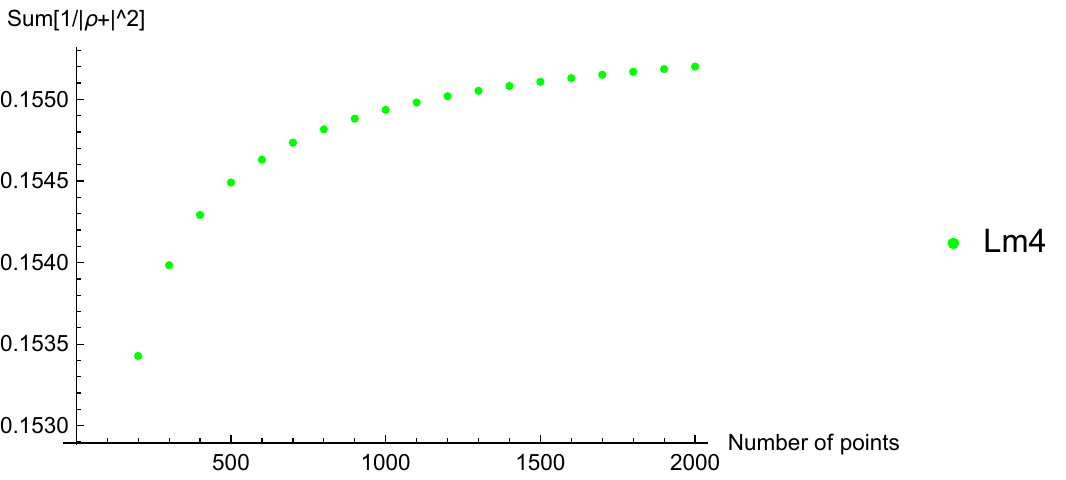}
\caption{Convergence of the sum over the inverse magnitude squared of zeros for
the functions $L_{-4}(s)$,  ${\tilde{\cal T}}_+(s)$ and $\xi(s)$ (left), and over an extended set of zeros for $L_{-4}(s)$ (right) . }
\label{fig-Lm1}
\end{figure}

The behaviour of the coefficients $\tau_k^{lm}$ and $\lambda_k^{lm}$ for $L_{-4}(s)$ is shown in Fig. \ref{fig-Lm2}.
There is an interesting difference between the behaviours of the $\tau$'s and $\lambda$'s for $\xi(s)$ and for $L_{-4}(s)$. In the fomer case the coefficients are monotonic, whereas in the latter case $\tau_k^{lm}$ is an oscillating function and while $\lambda_k^{lm}$ is generally increasing, it is no longer monotonic.
 \begin{figure}[tbh]
\includegraphics[width=7cm]{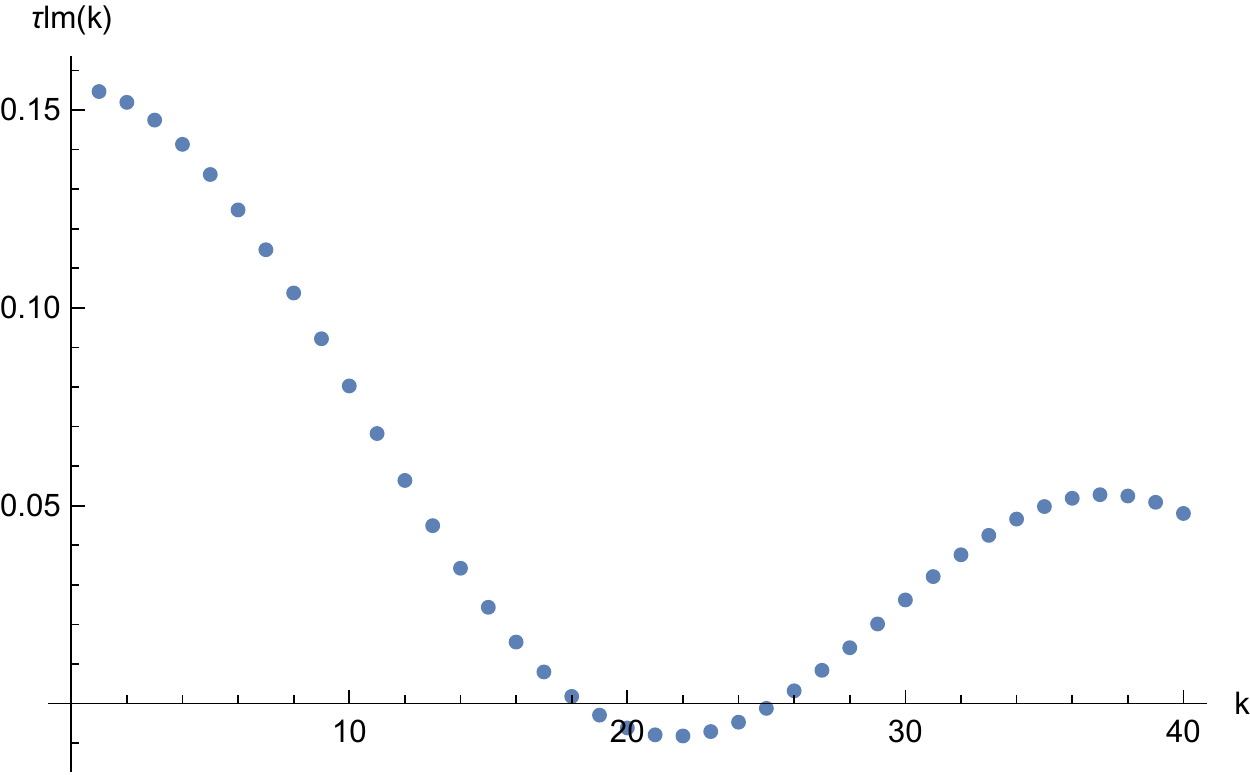}~\includegraphics[width=7cm]{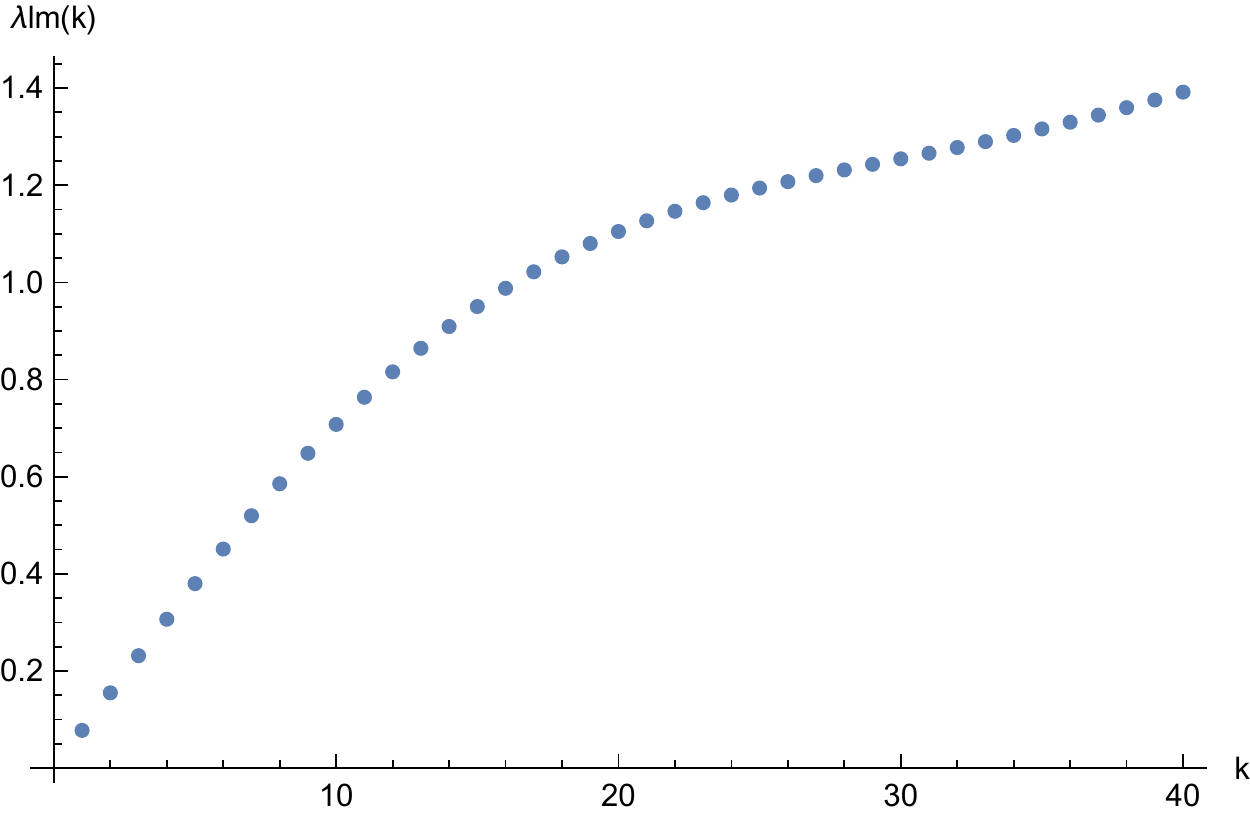}
\caption{the coefficients $\tau_k^{lm}$ and $\lambda_k^{lm}$ for $L_{-4}(s)$ as a function of $k$. }
\label{fig-Lm2}
\end{figure}
\section{The Link between  ${\tilde{\cal T}}_+(s)$, ${\tilde{\cal T}}_-(s)$ and $\xi(s)$}

The main motivation for this study is that the functions  ${\tilde{\cal T}}_+(s)$, ${\tilde{\cal T}}_-(s)$ and $\xi(s)$ are linked.
Indeed, from (\ref{T+0}) and (\ref{T-0}),
\begin{equation}
\xi_1(2 s)=2 [{\cal T}_+(s)+{\cal T}_-(s)],
\label{link-1}
\end{equation}
and
\begin{equation}
\xi_1(2 s-1)=2 [{\cal T}_+(s)-{\cal T}_-(s)].
\label{link-2}
\end{equation}
In terms of the functions  ${\tilde{\cal T}}_+(s)$ and  ${\tilde{\cal T}}_-(s)$ these are:
\begin{equation}
(1-s) \xi(2 s)=4 [{\tilde{\cal T}}_-(s)+(s-1/2) {\tilde{\cal T}}_+(s)],
\label{link-3}
\end{equation}
and
\begin{equation}
s \xi(2 s-1)=4 [{\tilde{\cal T}}_-(s)-(s-1/2) {\tilde{\cal T}}_+(s)].
\label{link-4}
\end{equation}
From \cite{ki}, \cite{lagandsuz} and \cite{mcp13}, the functions on the right-hand side of the last two equations above have all their non-trivial zeros on the critical line, all are simple, and occur alternately. The three functions have no singularities in the finite part of the complex plane. All the zeros of the left-hand side occur off the critical line, and their location and multiplicity are of course of much interest.

Given the additive nature of the relations (\ref{link-1}-\ref{link-4}), it is convenient to go from the representations  of functions in terms of
power series for their logarithm to power series for the functions themselves. To do this, we use the complex exponential Bell polynomials, for which there is a useful Wikipedia entry \cite{wiki}. The  complete exponential Bell polynomials occur in the exapansion:
\begin{equation}
\exp \left( \sum_{j=1}^\infty  x_j\frac{t^j}{j!}\right)=\sum_{n,k\ge 0} B_{n} (x_1,\ldots,x_{n})\frac{t^n}{n!}
\label{link-5}
\end{equation}
They satisfy the recurrence relation
\begin{equation}
B_{n+1} (x_1,\ldots,x_{n+1})=\sum_{i=0}^n { {n}\choose {i}}  B_{n-i} (x_1,\ldots,x_{n-i})x_{i+1}
\label{link-6}
\end{equation}
They have real coefficients, and the first five are:
\begin{eqnarray}
B_0=1,~~B_1(x_1)=x_1,&&B_2(x_1,x_2)=x_1^2+x_2, \nonumber\\
B_3(x_1,x_2,x_3)=x_1^3+3 x_1 x_2+x_3,& &B_4(x_1,x_2,x_3,x_4)=x_1^4+6 x_1^2 x_2+4 x_1 x_3+3 x_2^2+x_4.\nonumber \\
& &
\label{link-7}
\end{eqnarray}
They thus all begin with $x_1^n$ and end with $x_n$.

We apply (\ref{link-5}) to the generic expansion:
\begin{equation}
\log\left( \frac{F(s)}{F(0)}\right)=- \sum_{k=1}^\infty \left(\frac{1}{k}\right) \sigma_k^F s^k,
\label{link-8}
\end{equation}
with $\sigma_k^F$ denoting the sum over inverse $k$th powers of the zeros of $F$. We obtain then:
\begin{equation}
\frac{F(s)}{F(0)}=1+\sum_{k=1}^\infty B_k(-\sigma_1^F, -\sigma_2^F,\ldots,-\sigma_k^F (k-1)!) \frac{s^k}{k!}.
\label{link-9}
\end{equation}
The expansion given by (\ref{link-9}), up to order 4 in $s$, is 
\begin{eqnarray}
\frac{F(s)}{F(0)}&=&1-\sigma_1^F s+\left[\frac{(\sigma_1^F)^2-\sigma_2^F }{2}\right] s^2
  -\left[\frac{(\sigma_1^F)^3-3\sigma_1^F \sigma_2^F-2 \sigma_3^F }{6}\right] s^3 \nonumber\\
&&  +\left[\frac{(\sigma_1^F)^4-6(\sigma_1^F)^2 \sigma_2^F+8\sigma_1^F \sigma_3^F+3(\sigma_2^F)^2-6\sigma_4^F }{24}\right] s^4
  +\ldots
\label{link-10}
\end{eqnarray}

A second way of expressing (\ref{link-9}) is to use the infinite product involving its zeros:
\begin{equation}
\frac{F(s)}{F(0)}=\prod_k\left(1-\frac{s}{\rho_k^F}\right),
\label{link-11}
\end{equation}
where each term corresponds to one zero $\rho_k^F$ of $F(s)$. Expanding the product, we generate an expression for the function involving series coefficients formed from reciprocals of products involving non-identical permutations of the indices of the zeros for each term:
\begin{equation}
\frac{F(s)}{F(0)}=1-\left[\sum_k\frac{1}{\rho_k}\right]s+\left[\sum_{k\neq l}\frac{1}{\rho_k \rho_l}\right] s^2-\left[\sum_{k, l, m\neq}\frac{1}{\rho_k \rho_l \rho_m}\right] s^3+\ldots .
\label{link-12}
\end{equation}
Equation (\ref{link-10}) then gives an expression for each of the sums of products of zeros. For example, the term of order $s^2$ gives:
\begin{equation}
\frac{(\sigma_1^F)^2-\sigma_2^F}{2}=\sum_{k\neq l}\frac{1}{\rho_k \rho_l}.
\label{link12a}
\end{equation}
The approximate value of the left-hand side in (\ref{link12a}) for: $\xi(s)$ is $0.0233439$, for ${\tilde{\cal T}}_+(s)$ is 0.0974409 and for 
${\tilde{\cal T}}_-(s)$ is $-0.0050851$. The last of these is of course influenced by the two real  zeros off the critical line.

We now use (\ref{link-9}) to express the power series for the three functions occurring in (\ref{link-3}).
Firstly, for $\xi(2 s)$:
\begin{equation}
(1-s) \xi(2 s)=\frac{1}{2}(1-s)\left[1+\sum_{k=1}^\infty B_k(-\sigma_1^K, -\sigma_2^K,\ldots,-\sigma_k^K (k-1)!) \frac{(2 s)^k}{k!}\right].
\label{link-13}
\end{equation}
Next, for ${\tilde{\cal T}}_+(s)$ and  ${\tilde{\cal T}}_-(s)$:
\begin{equation}
4(s-\frac{1}{2}) {\tilde{\cal T}}_+(s)=-\frac{1}{2}(s-\frac{1}{2})\left[1+\sum_{k=1}^\infty B_k(-\sigma_1^+, -\sigma_2^+,\ldots,-\sigma_k^+ (k-1)!) \frac{ s^k}{k!}\right],
\label{link-14}
\end{equation}
and
\begin{equation}
4{\tilde{\cal T}}_-(s)=\frac{1}{4}\left[1+\sum_{k=1}^\infty B_k(-\sigma_1^-, -\sigma_2^-,\ldots,-\sigma_k^- (k-1)!) \frac{ s^k}{k!}\right],
\label{link-15}
\end{equation}

Table \ref{tabco} gives the first six coefficients of powers of $s$ in these three expansions, as well as the sum of the coefficients corresponding to the right-hand side  of (\ref{link-3}). The contributions from the right-hand side are dominated by those from ${\tilde{\cal T}}_+(s)$.
The differences between the coefficients on the left- and right-hand sides
are below $2 \times 10^{-14}$ in magnitude, the default level of accuracy in Mathematica.
\begin{table}
\begin{tabular}{|c|c|c|c|c|c|c|}\\ \hline
Coefficient&0&1&2&3&4&5\\ \hline
lhs(\ref{link-3})&1/2 &  -0.523096 & 0.0697834& -0.0486797&0.00401739 & -0.00210626 \\
rhs(\ref{link-3})-1st&1/4 & 0.000251533 & -0.00127128 &0.00209431& -0.0011858& 0.000170825\\
rhs(\ref{link-3})-2nd& 1/4 &-0.523347 & 0.0710547 &-0.050774 &0.00520319 &-0.00227708\\
total (rhs) &1/2 & -0.523096 &0.0697834 &-0.0486797&0.00401739& -0.00210626\\ \hline
\end{tabular}
\caption{The coefficients of the various powers of $s$ in the three terms occurring in equation (\ref{link-3}).} 
\label{tabco}
\end{table}

 Using  (\ref{link-9})  and (\ref{link-13}-\ref{link-15}), the relation coming from the coefficient of $s$ is
 \begin{equation}
 \sigma_1^K=\frac{1}{4}( \sigma_1^++ \sigma_1^-).
 \label{link-16}
 \end{equation}
 The relation coming from the coefficient of $s^k$ for $k\ge 2$ is
  \begin{eqnarray}
&&2^k \left[B_k (-\sigma_1^K, -\sigma_2^K,\ldots,-\sigma_k^K (k-1)!)-\frac{k}{2} B_{k-1} (-\sigma_1^K, -\sigma_2^K,\ldots,-\sigma_{k-1}^K (k-2)!)- \right]= \nonumber\\
&&\frac{1}{2} \left[B_k (-\sigma_1^+, -\sigma_2^+,\ldots,-\sigma_k^+ (k-1)!)+B_k (-\sigma_1^-, -\sigma_2^-,\ldots,-\sigma_k^- (k-1)!))
\right. \nonumber\\
&& \left. -k B_{k-1} (-\sigma_1^+, -\sigma_2^+,\ldots,-\sigma_{k-1}^+ (k-2)!) \right]
\label{link-17}
\end{eqnarray}
For example, for $k=2$, and using (\ref{link-16}),
\begin{equation}
\sigma_2^K=\frac{1}{16} [-(\sigma_1^-)^2+2 \sigma_2^- -4 \sigma_1^+  - (\sigma_1^+)^2+2  \sigma_1^-(2+ \sigma_1^++2  \sigma_2^+)].
\label{link-18}
\end{equation}
For larger values of $k$, the expressions rising from (\ref{link-17}) increase rapidly in complexity. This renders more difficult the task of trying to establish relationships between the $\tau^K$ and the $\tau^+$, $\tau^-$, which remains an unachieved goal of this work.
\section{Translations of Functions and Zeros}
From the work of Lagarias and Suzuki \cite{lagandsuz}, we  know that the following two functions have all their zeros on the critical line:
\begin{equation}
\frac{\xi_1 (2 s)}{s-1}-\frac{\xi_1 (2-2 s)}{s}
\label{trans1}
\end{equation}
and, for each fixed $T\geq 1$,
\begin{equation}
\frac{-\xi_1 (2 s) T^{s-1}}{s-1}+\frac{\xi_1 (2-2 s) T^{-s}}{s}.
\label{trans2}
\end{equation}
Also, for each $y\ge 1$,
\begin{equation}
\xi_1 (2 s) y^{s}+\xi_1 (2-2 s) y^{1-s}
\label{trans3}
\end{equation}
has all its zeros on the critical line for $1\le y\le y_*\approx 7.055507$, and exactly two off the critical line for $y>y_*$.

An obvious addition to this list is functions obtained from one having all its zeros on the critical line by translations and rescalings will have all their zeros on straight lines in the complex plane:
\begin{equation}
f(s)\rightarrow f(s-s_0)\rightarrow f(\alpha (s-s0)),
\label{trans4}
\end{equation}
for constants $s_0$ and $\alpha$. In this section, we will discuss translations, i.e. $\alpha$ will be kept as unity.

Consider then the effect of a translation on a function $f(s)$ having the product representation (\ref{prod7}). Putting $\hat{z}=z-z_0$,
the zeros $a_n$ become $\hat{a}_n=a_n-z_0$. The inverse sums of powers of zeros may then be evaluated as a function of $z_0$:
\begin{equation}
\hat{\sigma}_m=\sum_{p=0}^\infty  {} ^{-m}C_p (-z_0)^p \sigma_{m+p} =\sum_{p=0}^\infty \frac{\Gamma (m+p)}{\Gamma (p+1) \Gamma (m-1)}  z_0^p \sigma_{m+p}.
\label{trans5}
\end{equation}
Using (\ref{prod6}), (\ref{trans5}) may be rewritten as
\begin{equation}
\hat{\sigma}_m=\left.-\frac{1}{\Gamma (m-1)} \frac{d^m}{d z^m} \log f(z)\right|_{z=z_0}.
\label{trans6}
\end{equation}
Note that, even if the translation is along the critical line, and all the quantities $\sigma_m$ are real, the quantities $\hat{\sigma}_m$ are in general complex, as the translation changes  the link between zeros:  $a_m-z_0$ is linked to $1-a_m-z_0$.

Translations along the critical line are however of use in modifying relationships between zeros of different functions. For example, zeros of ${\cal T}_+(s)$ and  ${\cal T}_-(s)$ all lie on the critical line and alternate. It is interesting to investigate whether any similar relationship exists between the zeros of either and those of $\xi_1(2 s-1/2)$, which by the Riemann hypothesis all lie on the critical line.
We have investigated firstly whether zeros of $\xi_1(2 s-1/2)$ all lie after those of ${\cal T}_+(s)$. In fact, of the first 1500 zeros of the latter, this property does not hold in 232 or 15.5\% cases. We can also ask whether  zeros of $\xi_1(2 s-1/2)$  lie between successive zeros of ${\cal T}_-(s)$. This property fails only in four cases: 921, 995, 1307 and 1495. Using translations we can make the property hold for all 1500 zeros: translations $s_0=i t_0$ achieve this for $t_0$ in the range -0.080 to -0.036. This same translation along the critical line leads to a variation along it of the shifted ratio ${\cal T}_+(s)/{\cal T}_-(s)$ of the form
\begin{equation}
\frac{{\cal T}_+(1/2+i (t-t_0))}{{\cal T}_-(1/2+i (t-t_0))}=-i \cot [\arg \xi_1(1+2 i (t-t_0))].
\label{trans7}
\end{equation}
Hence, each point on the critical line can be made a zero of either the numerator or the denominator function in (\ref{trans7}). In addition,
each zero of $\xi_1(2 s-1/2)$ on the critical line can be made to coincide with a zero of ${\cal T}_-(s-s_0)$ or of ${\cal T}_+(s-s_0)$, or to lie between two such.

\section{An Argument in Support of the Riemann Hypothesis}

We now consider further properties of the functions ${\cal T}_+(s)$ and ${\cal T}_-(s)$. As commented above, it has been 
proved \cite{prt,lagandsuz,ki} that both of these functions have all their non-trivial zeros on the critical line, and that the zeros of the two functions interlace there. The celebrated Riemann hypothesis is that $\zeta (2 s-1/2)$ has all its non-trivial zeros on the critical line,
while numerical evidence has been presented \cite{mcp13} that the distribution functions of the zeros of all three of these functions agree in terms which remain finite as their argument tends to infinity on the critical line. We discuss in this section new analytical and graphical arguments which support the Riemann hypothesis.
 \begin{figure}[tbh]
\includegraphics[width=7cm]{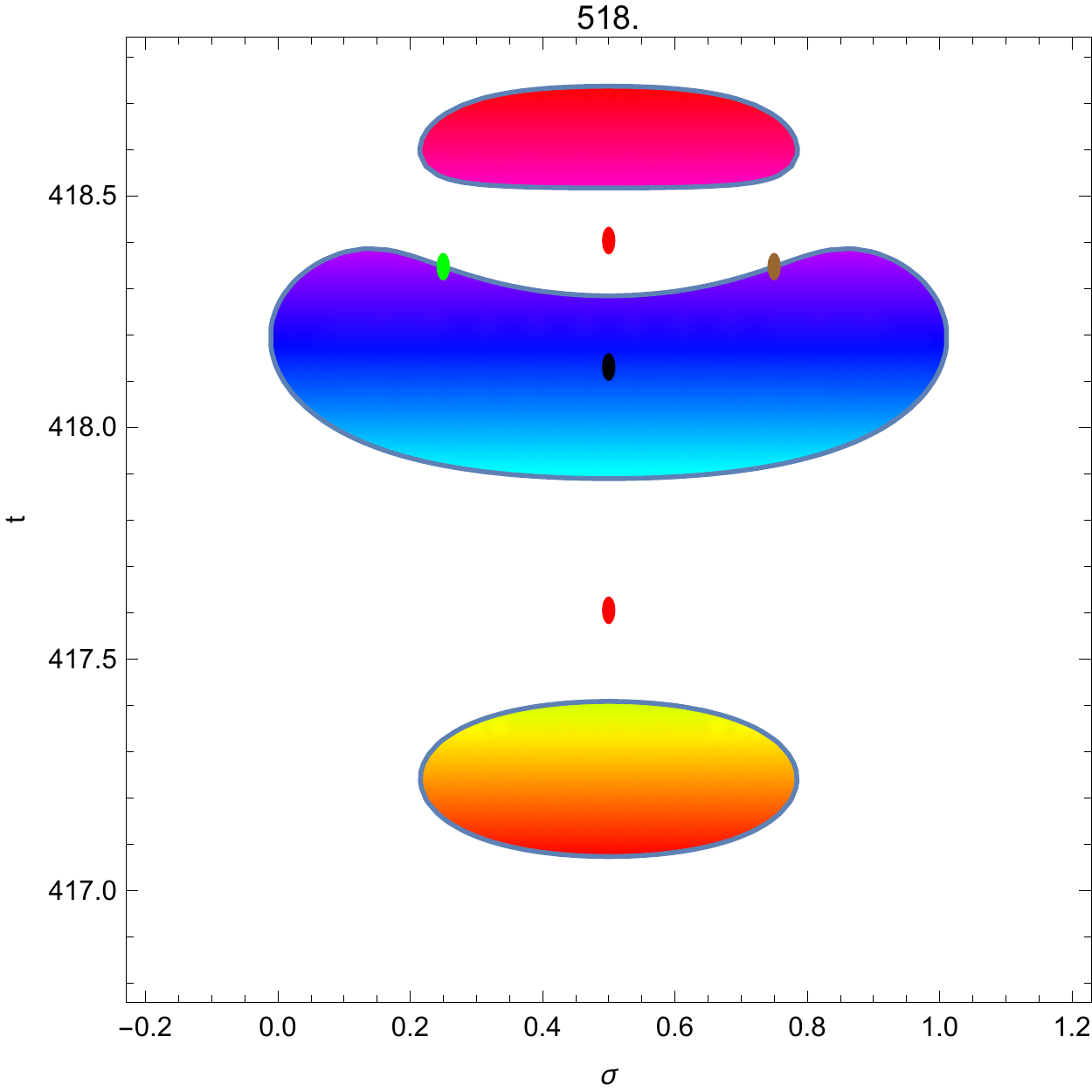}~\includegraphics[width=7cm]{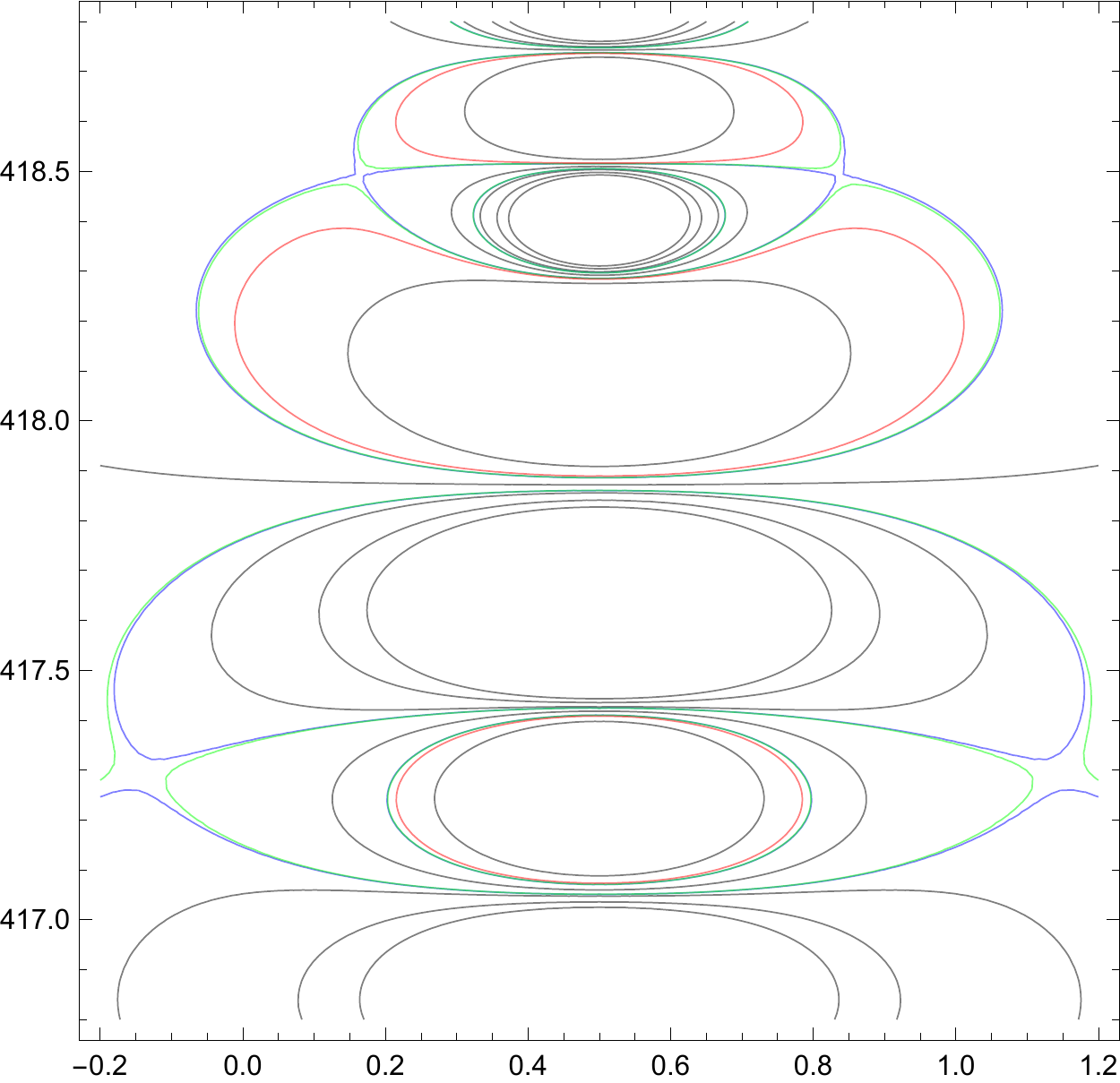}
\caption{At left, three successive regions (coloured)  in which $|{\cal V}(\sigma+i t)|>0$ are shown. Black dots represent zeros, red dots poles, while green and brown dots denote zeros of $\xi_1( 2s)$ and $\xi_1(2 s-1)$. At right, corresponding contours of constant modulus, with red denoting unit modulus, while green and blue contours correspond to moduli just above and below the values corresponding to the derivative zeros of ${\cal V}(\sigma+i t)$. }
\label{fig-srh1}
\end{figure}

 \begin{figure}[tbh]
\includegraphics[width=7cm]{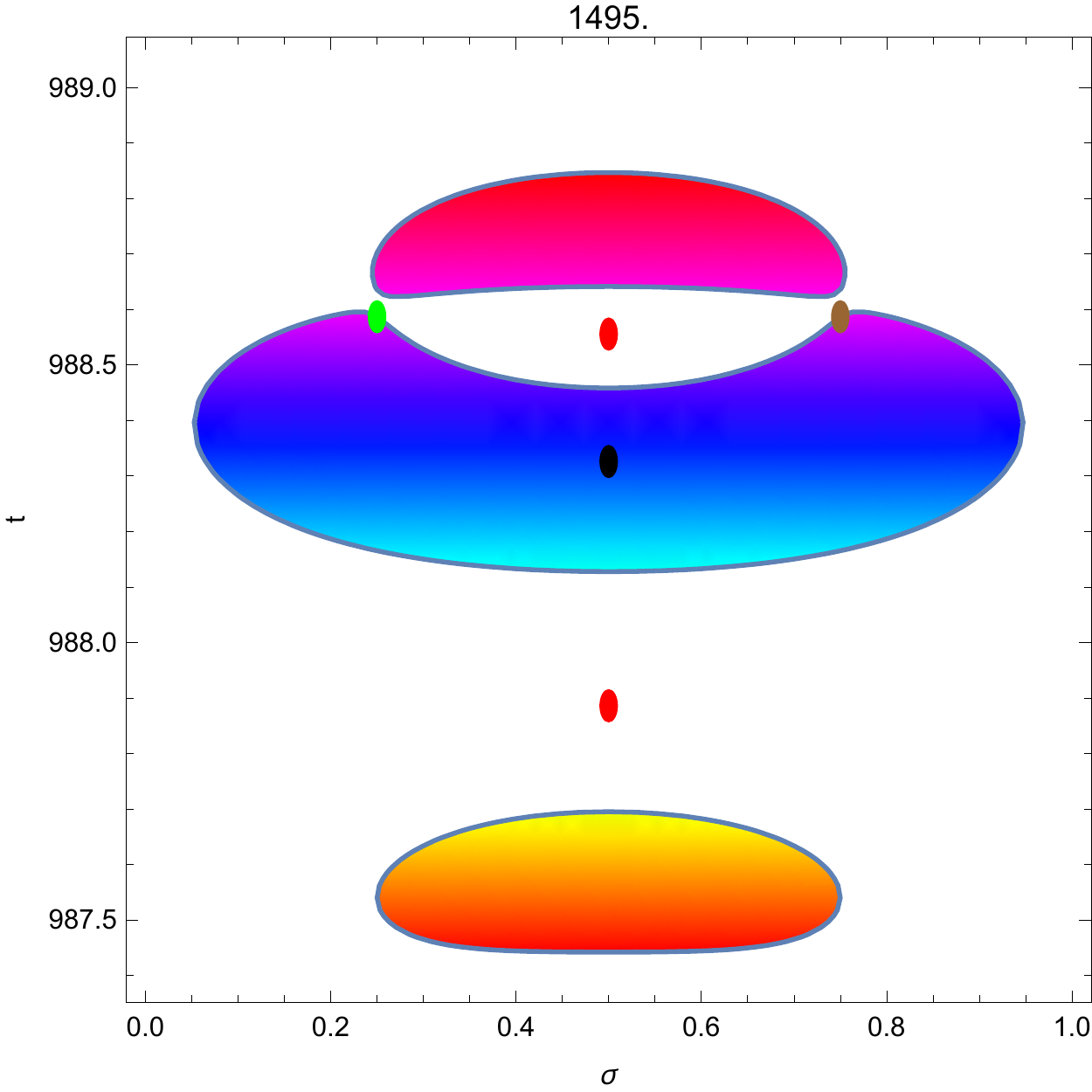}~\includegraphics[width=7cm]{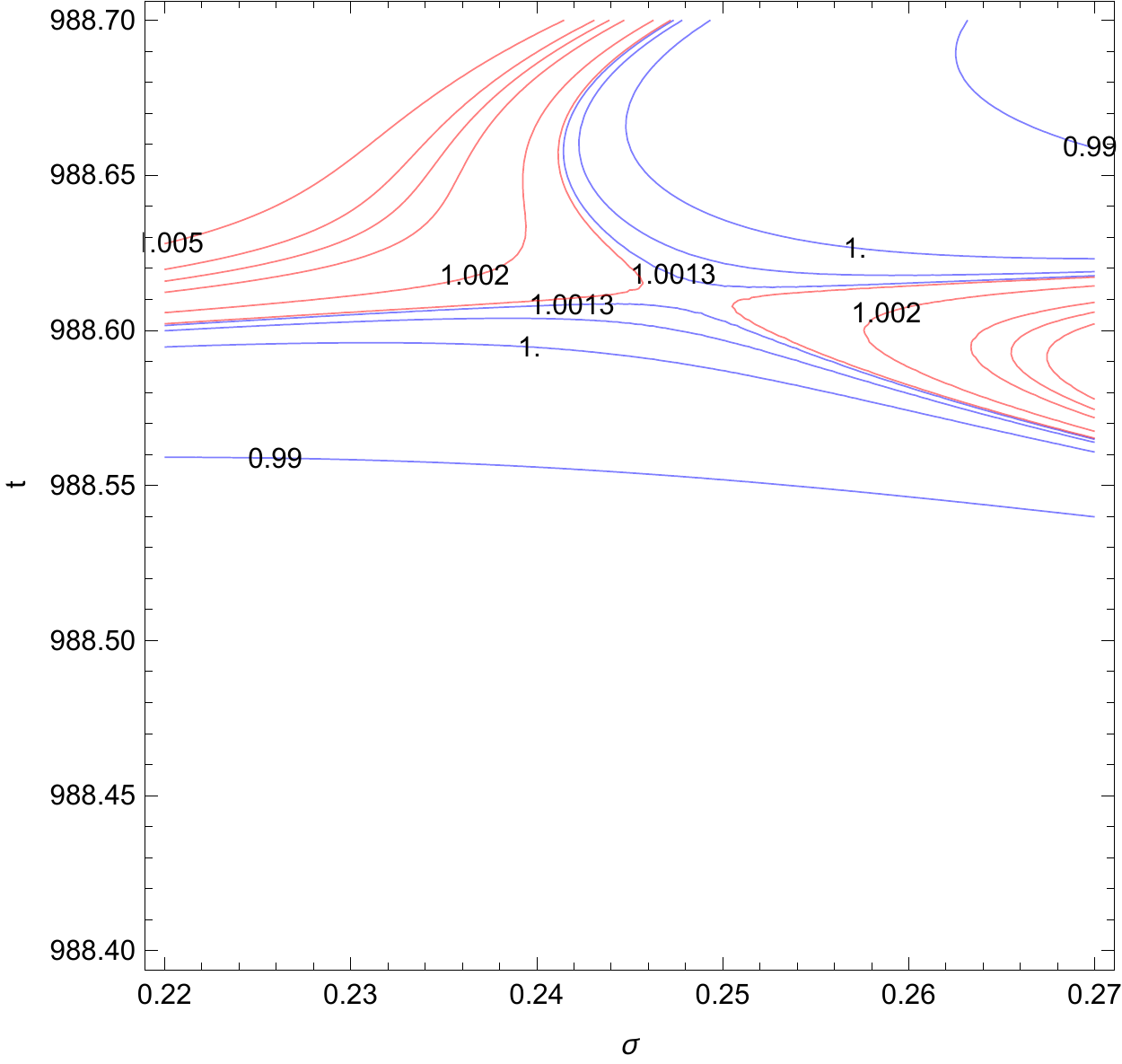}
\includegraphics[width=7cm]{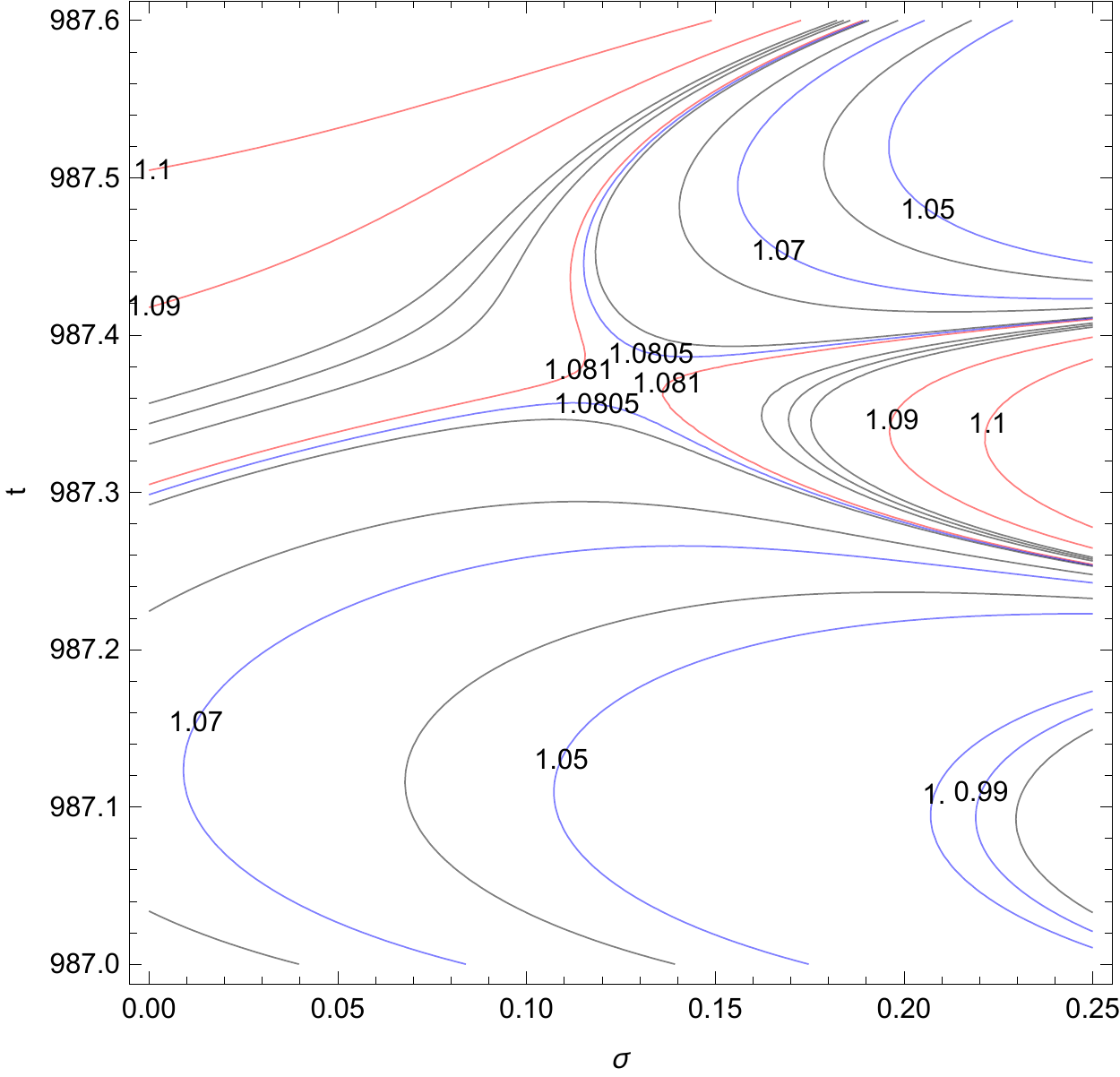}
\caption{At left, three successive regions (coloured)  in which $|{\cal V}(\sigma+i t)|>0$ are shown. Black dots represent zeros, red dots poles, while green and brown dots denote zeros of $\xi_1( 2s)$ and $\xi_1(2 s-1)$. At right, corresponding contours of constant modulus, showing the region around one of the  derivative zeros of ${\cal V}(\sigma+i t)$. Below: detail of contours of constant modulus in the region of the lower derivative zero.}
\label{fig-srh2}
\end{figure}

From the functions ${\cal T}_+(s)$ and ${\cal T}_-(s)$ we construct two further functions:
\begin{equation}
{\cal V}(s)=\frac{{\cal T}_+(s)}{{\cal T}_-(s)}=\frac{1+{\cal U}(s)}{1-{\cal U}(s)},
\label{eq-srh1}
\end{equation}
where
\begin{equation}
{\cal U}(s)=\frac{\xi_1(2s -1)}{\xi_1(2 s)}.
\label{eq-srh2}
\end{equation}
Note that
\begin{equation}
{\cal V}(s)=\frac{1+\sqrt{\pi} \Gamma (s-1/2)\zeta (2 s-1)/(\Gamma (s)\zeta (2 s))}{1-\sqrt{\pi} \Gamma (s-1/2)\zeta (2 s-1)/(\Gamma (s)\zeta (2 s))} ,
\label{eq-srh3}
\end{equation}
leading to the first-order estimate for $|{\cal V}(s)|$ for $1<<|\sigma|<<t$
\begin{equation}
|{\cal V}(s)|\sim 1+\sqrt{\frac{2}{t}}.
\label{eq-srh4}
\end{equation}
The corresponding argument estimate in  $1<<\sigma<<t$ is
\begin{equation}
\arg [{\cal V}(s)] \sim -\sqrt{\frac{2 \pi}{t}}.
\label{eq-srh5}
\end{equation}

We know that ${\cal V}(s)$ has all its non-trivial zeros and poles interlaced along the critical line, while the Riemann hypothesis is that
${\cal U}(s)$ has its zeros on $\Re (s)=3/4$ and its poles on $\Re (s)=1/4$. From (\ref{eq-srh1}), zeros of ${\cal U}(s)$ correspond 
to ${\cal V}(s)=1$ and poles to ${\cal V}(s)=-1$. Both then must lie on contours  of constant modulus $|{\cal V}(s)|=1$, which correspond to ${\cal U}(s)$ being pure imaginary: $\Re [{\cal U}(s)]=0$.

An investigation has been carried out into the relationship between the contours of constant modulus $|{\cal V}(s)|=1$ and the location of the zeros and poles of ${\cal U}(s)$, for the first 1500 zeros of ${\cal T}_+(s)$ and ${\cal T}_-(s)$. A convenient way of doing this in the symbolic/numerical/graphical package Mathematica is to use the option RegionPlot, and in this case to construct the regions in which $|{\cal V}(s)|\leq 1$,  or equivalently in which $\Re [{\cal U}(s)]<0$. These can be combined with contour plots of  $|{\cal V}(s)|$,
with contours appropriately chosen to highlight the location and behaviour around zeros of the derivative function ${\cal U}'(s)$, evaluated by numerical differentiation.

The results of this (rather labour intensive) investigation are quite suggestive. In each of the 1500 cases, a zero of ${\cal T}_+(s)$
on the critical line sits at the centre of a simply-connected region, whose boundary fully encloses the region $|{\cal V}(s)|\leq 1$.
The region $|{\cal V}(s)|> 1$ is multiply connected, in keeping with the estimate in equation (\ref{eq-srh3}) for $\sigma$ not too close to the critical line. Two examples are given in Figs. \ref{fig-srh1} and  \ref{fig-srh2}. The first example (zero 518 of ${\cal T}_+(s)$) of shows what may be described as  typical behaviour, while the second (zero 1495)  corresponds to one of the four exceptions mentioned in the previous section.

In the first example, the contours of constant modulus shown are for the levels 0.90, 1.0, 1.018, 1.019, 1.1, 1.169, 1.170, 1.2, 1.3, 1.4.
The zeros of ${\cal U}'(s)$, or equivalently of ${\cal V}'(s)$, are approximately  $s=-0.143103+417.293 i$, where $|{\cal V}(s)|=1.16957$,
and $s=0.163301+ 418.4092 i$, where $|{\cal V}(s)|=1.01891$. The upper derivative zero is defined by four contours of constant modulus, two in green provided by the zeros of ${\cal V}(s)$, and two in blue, one pertaining to the intervening pole of ${\cal V}(s)$ and the other to a closed curve enclosing the two poles and one zero. The lower structure is not complete as shown, but the outermost curve encloses two poles and an intervening zero.

In the second example, the zeros of ${\cal U}'(s)$ are approximately
$s=0.24809+988.611 i$, where $|{\cal V}(s)|=1.001357$, and $s=0.12566+ 987.373 i$, where $|{\cal V}(s)|=1.0808$. The structure near the upper derivative zero could be described as close to closed, with the two contours of constant modulus coming close to touching.
In consequence, the modulus at the derivative zero is much closer to unity than for the far more open structure around the lower derivative zero. For this example, both derivative zeros referred to are provided by two zeros of ${\cal V}(s)$ surrounding an intervening pole. 

We now commence the analytical exploration of these numerical results. The aim is to say as much as possible in support of the validity of the Riemann hypothesis, using as a tool the results already referred to about the properties of the functions ${\cal T}_+(s)$ and 
${\cal T}_-(s)$, or ${\cal U}(s)$ and ${\cal V}(s)$.

{\bf Remark:} The only possible closed contours whose boundary is an equimodular contour of ${\cal V}(s)$ in $t$ not small  cut the critical line.
This is a simple consequence of the Maximum/Minimum Modulus Theorems, since the only non-trivial zeros and poles of ${\cal V}(s)$ 
lie on the critical line. Furthermore, contours of constant modulus touching the critical line are precluded since ${\cal V}'(s)$ is never zero on the critical line.

\begin{theorem}
If all  non-trivial zeros of ${\cal V}'(s)$ correspond to a modulus $|{\cal V}(s)|>1$, then there exists a set  of simple zeros of ${\cal U}((s)$ off the critical line in one-to-one correspondence with the zeros of ${\cal T}_+(s)$ on the critical line.
\label{suffthm}
\end{theorem}
\begin{proof}
The proof generalises the reasoning of Macdonald  \cite{macd} to functions having poles and zeros. We note that 
$|{\cal V}(s)|=|{\cal V}(1-\overline{s})|$ is a symmetric function under reflection in the critical line, so curves of constant modulus share this property. Starting from a general zero of ${\cal T}_+(s)$, we constrain the family of closed curves on which its modulus is constant.
This family of curves has as its final member that curve of constant modulus touching a zero of ${\cal V}'(s)$, and so by the assumption of this Theorem the corresponding constant modulus exceeds unity. It then follows that there is a closed curve of constant modulus unity enclosing the zero of ${\cal T}_+(s)$. This closed curve intersects the critical line at points where ${\cal V}(s)={\cal U}(s)=\pm i$,
and at every point on it ${\cal U}(s)$ is pure imaginary. The curve encloses one zero of ${\cal V}(s)$, and thus the argument of this function increases monotonically through a range of $2\pi$ along it, ensuring that it passes through $\pm 1$ along it. On each such curve of constant modulus unity then there is a simple pole and a simple zero of ${\cal U}(s)$, establishing the one-to-one correspondence referred to in the Theorem statement.
\end{proof}

\begin{corollary}
If all  non-trivial zeros of ${\cal V}'(s)$ correspond to a modulus $|{\cal V}(s)|>1$, then the distribution function for zeros of $\xi_1(2 s-1/2)$ on the critical line with $0<t<T$ is
\begin{equation}
\frac{T}{\pi} \log\left(\frac{T}{\pi} \right)-\frac{T}{\pi}  .
\label{eq-srh6}
\end{equation}
\end{corollary}
\begin{proof}
From Theorem \ref{suffthm}, there exists for each zero of ${\cal T}_+(s)$ on the critical line a simple zero of ${\cal U}(s)$ lying on the contour $|{\cal T}_+(s)/{\cal T}_-(s)|$ enclosing the zero of ${\cal T}_+(s)$. Now the zeros of ${\cal T}_+(s)$ are all simple, and in 1:1 correspondence with the lines of a specified argument coming from $\sigma>>1$, and passing through the zero on $\sigma=1/2$.
We have the following asymptotic estimates for ${\cal T}_+(s)$, ${\cal T}_-(s)$ in $t>>\sigma>>1$:
\begin{eqnarray}
\left. \begin{tabular}{c}
${\cal T}_+(s)$\\
${\cal T}_-(s)$
\end{tabular}
\right\} &= & \exp\left[-\left(\sigma-\frac{1}{2}\right) \log 2\pi -\frac{\pi t}{2}+\left(\sigma-\frac{1}{2}\right) \log t\right]  \nonumber\\
& & \times   \exp\left[ i\left(\frac{\pi}{2}\left(\sigma-\frac{1}{2}\right) -t\left(1+\log\left(\frac{\pi}{t}\right)\right)\right)\right] \left( 1\pm i\sqrt{\frac{2 \pi}{t}}\right) \nonumber\\
& &  .
\label{eq-srh7}
\end{eqnarray}
This then gives the distribution function (\ref{eq-srh6}) for zeros  of ${\cal T}_+(1/2+ i t)$ lying between $0$ and $T$.
Now, this is precisely the formula \cite{edw} for the number of zeros of $\xi_1( 2 s-1/2)$ lying in the strip $1/4<\sigma<3/4$ with $0<t<T$. If any such zero were to lie off the critical line, it would have to occur in a pair of zeros symmetric about the critical line. However, we have established that there exists the set of simple zeros of $\xi_1(2 s)$ in precise 1:1 correspondence with the zeros of  ${\cal T}_+(1/2+ i t)$. Hence, any zeros off the critical line would have to be sufficiently rare to leave the distribution function (\ref{eq-srh6}) unaltered. In other words, the distributions functions for zeros on the critical line of $\xi_1( 2 s-1/2)$, ${\cal T}_+(s)$  and ${\cal T}_-(s)$ are all the same and given by (\ref{eq-srh6}),
if Theorem \ref{suffthm} holds. 
\end{proof}

There is a nice duality inherent in the previous Theorem and Corollary. Each contour $|{\cal V}(s)|=1$  corresponds to points on the critical line where (going from smaller to larger $t$) ${\cal V}(s)={\cal U}(s)=-i$ and then ${\cal V}(s)={\cal U}(s)=i$. Along the critical line (where ${\cal U}(s)$ is pure imaginary),
$\arg {\cal V}(s)$ goes from $-\pi/2$ below the zero of ${\cal T}_+(s)$ to $\pi/2$  above it.  Going around $|{\cal V}(s)|=1$ (i.e. where ${\cal U}(s)$ is pure imaginary) in the anti-clockwise sense, $\arg {\cal U}(s)$ goes from $-\pi/2$ before  the zero of $\xi_1(2 s-1)$ to $\pi/2$  after it.

We now investigate the condition upon which Theorem \ref{suffthm} holds: all  non-trivial zeros of ${\cal V}'(s)$ correspond to a modulus $|{\cal V}(s)|>1$. The argument of the Theorem without this assumption being made leads to an association between zeros and poles of ${\cal V}(s)$ and zeros of the derivative  ${\cal V}'(s)$. More specifically, each zero of  ${\cal V}'(s)$, say $s_d$ is associated with a triplet: either $ZPZ$, two zeros sandwiching an intervening pole, or $PZP$, two poles sandwiching an intervening zero. The condition for Theorem \ref{suffthm} is necessary to show that the zero and pole of ${\cal U}(s)$ lie on the curve(s) of unit modulus perforce associated with zero(s) of ${\cal V}(s)$, and not with pole(s). Let us consider three successive triplets: previous (associated with $s_{dp}$),
current (associated with $s_{dc}$) and subsequent (associated with $s_{ds}$).

\begin{theorem}
If $|{\cal V}(s_{dc})|>1$ for a system $ZPZ$, then $|{\cal V}(s_{dp})|>1$ and $|{\cal V}(s_{ds})|>1$.\\
If $|{\cal V}(s_{dc})|<1$ for a system $PZP$, then $|{\cal V}(s_{dp})|<1$ and $|{\cal V}(s_{ds})|<1$.
\label{condthm}
\end{theorem}
\begin{proof}
Given a system $ZPZ$ with $|{\cal V}(s_{dc})|>1$. Then the system is enclosed by a curve of constant modulus larger than unity,
and has before it and after it on the critical line systems $PZP$. As the modulus of ${\cal V}(s)$ increases monotonically on the critical line going from a zero towards an adjacent pole, the curves of constant modulus bounding the previous and subsequent systems $PZP$ must also correspond to moduli exceeding unity.

Given a system $PZP$ with $|{\cal V}(s_{dc})|<1$. Then the system is enclosed by a curve of constant modulus smaller than unity,
and has before it and after it on the critical line systems $ZPZ$. As the modulus of ${\cal V}(s)$ decreases monotonically on the critical line going from a pole towards an adjacent zero, the curves of constant modulus bounding the previous and subsequent systems $ZPZ$ must also correspond to moduli smaller than unity.
\end{proof}

This argument provides a justification for  the assumption underlying Theorem \ref{suffthm}.

\end{document}